\documentclass[12pt]{amsart}
\usepackage[top=1in, left=1in, right=1in, bottom=1in]{geometry} %
\usepackage{latexsym, amsmath, amscd, amsthm}
\usepackage{mathrsfs}
\usepackage{enumerate}

\usepackage[utf8]{inputenc}
\usepackage[
uniquename=false,
backend=biber,
sortcites=true,
style=alphabetic,
giveninits=true]{biblatex}
\usepackage[svgnames]{xcolor}
\usepackage{tikz}
\usetikzlibrary{decorations.markings,backgrounds,hobby,knots,calc}
\usepackage{pgfplots}\pgfplotsset{compat=1.13}
\usepackage{graphicx}
\usepackage{subcaption}

\usepackage[mode=image]{standalone}

\usepackage{hyperref}

\newcommand{\Z}{\mathbb{Z}} \newcommand{\N}{\mathbb{N}}

\newcommand{\R}{\mathbb{R}}

\newcommand{\FlatKnotDia}{\mathscr{K}}
\newcommand{\KnotShad}{\FlatKnotDia}

\newcommand{\Exp}{\mathbb{E}}
\newcommand{\Prb}{\pi}
\newcommand{\Trprb}{P}

\DeclareMathOperator{\Aut}{aut}

\DeclareMathOperator{\tp}{tp}

\DeclareFieldFormat{url}{%
  \iffieldundef{doi}{%
    \mkbibacro{URL}\addcolon\space\url{#1}%
  }{%
  }%
}

\DeclareFieldFormat{urldate}{%
  \iffieldundef{doi}{%
    \mkbibparens{\bibstring{urlseen}\space#1}%
  }{%
  }%
}

\makeatletter

\newrobustcmd*{\parentexttrack}[1]{%
  \begingroup
  \blx@blxinit
  \blx@setsfcodes
  \blx@bibopenparen#1\blx@bibcloseparen
  \endgroup}

\AtEveryCite{%
  }

\makeatother

\DeclareMathOperator{\RI}{RI}
\DeclareMathOperator{\RII}{RII}
\DeclareMathOperator{\RIII}{RIII}

\newcommand{\RIp}{\RI^+}
\newcommand{\RIm}{\RI^-}

\newcommand{\RIIp}{\RII^+}
\newcommand{\RIIm}{\RII^-}

\usepackage{contour}

\makeatletter
\newtheorem*{rep@theorem}{\rep@title}
\newcommand{\newreptheorem}[2]{%
\newenvironment{rep#1}[1]{%
 \def\rep@title{#2 \ref{##1}}%
 \begin{rep@theorem}}%
 {\end{rep@theorem}}}
\makeatother

\newtheorem{theorem}{Theorem}
\newreptheorem{theorem}{Theorem}

\newtheorem*{lemma*}{Lemma}
\newtheorem{corollary}[theorem]{Corollary}

\pgfkeys{
/pgf/decoration/.cd,
pre fraction/.style={pre length=#1*\pgfmetadecoratedpathlength},
post fraction/.style={post length=#1*\pgfmetadecoratedpathlength}
}

\tikzset{->-/.style={decoration={ markings, mark=at position #1
      with {\arrow{>}}},postaction={decorate}}}

\newcommand{\arcthick}{4}
\newcommand{\arclenfrac}{1.2}
\newcommand{\arclenfracarc}{0.2}

\tikzset{startarc/.style n args={2}{
    decoration={curveto, raise=#1,
    post=moveto, post fraction=\arclenfracarc},
    preaction={decorate,draw,#2,line width=\arcthick pt}}}

\tikzset{endarc/.style n args={2}{
    decoration={curveto, raise=#1,
    pre=moveto, pre fraction=\arclenfracarc},
    preaction={decorate,draw,#2,line width=\arcthick pt}}}

\tikzset{starthalfarc/.style n args={3}{
  decoration={curveto, raise=#1,
    post=moveto, post fraction=#2*\arclenfrac},
  decorate,
  line width=\arcthick pt,
  draw=#3}}

\tikzset{endhalfarc/.style n args={3}{
  decoration={curveto, raise=#1,
    pre=moveto, pre fraction=#2*\arclenfrac},
  decorate,
  line width=\arcthick pt,
  draw=#3}}

\tikzset{fwdarcs/.style n args={4}{
    preaction={starthalfarc={#1}{#2}{#3}},
    preaction={endhalfarc={#1*-1}{#2}{#4}}}}

\newcounter{marknumber}
\pgfplotsset{
    error bars/every nth mark/.style={
        /pgfplots/error bars/draw error bar/.prefix code={
            \pgfmathtruncatemacro\marknumbercheck{mod(floor(\themarknumber/2),#1)}
            \ifnum\marknumbercheck=0
            \else
                \begin{scope}[opacity=0]
            \fi
        },
        /pgfplots/error bars/draw error bar/.append code={
            \ifnum\marknumbercheck=0
            \else
                \end{scope}
            \fi
            \stepcounter{marknumber}    
        }
      }
    }

\newcounter{marknumberx}
\pgfplotsset{
    error bars/every nth markx/.style={
        /pgfplots/error bars/draw error bar/.prefix code={
            \pgfmathtruncatemacro\marknumberxcheck{mod(floor(\themarknumberx/2),#1)}
            \ifnum\marknumberxcheck=0
            \else
                \begin{scope}[opacity=0]
            \fi
        },
        /pgfplots/error bars/draw error bar/.append code={
            \ifnum\marknumberxcheck=0
            \else
                \end{scope}
            \fi
            \stepcounter{marknumberx}    
        }
      }
    }

\newif\iflegend
\legendfalse

\newif\ifcustomxlabel
\customxlabelfalse

\newif\ifcustomtitle
\customtitlefalse

\newif\ifallmaps
\allmapstrue

\begin{document}
\title[]{A Markov Chain sampler for plane curves}
\author{Harrison Chapman}
\email{hchaps@gmail.com}
\address{Department of Mathematics\\
  Colorado State University, Fort Collins CO}
\author{Andrew Rechnitzer}
\email{andrewr@math.ubc.ca}
\address{Department of Mathematics\\
  University of British Columbia, Vancouver BC}
\date{\today}

\begin{abstract}
  A plane curve is a knot diagram in which each crossing is replaced by a 4-valent vertex, and so are dual to a subset of planar quadrangulations. The aim of this paper is to introduce a new tool for sampling diagrams via sampling of plane curves. At present the most efficient method for sampling diagrams is rejection sampling, however that method is inefficient at even modest sizes. We introduce Markov chains that sample from the space of plane curves using local moves based on Reidemeister moves. By then mapping vertices on those curves to crossings we produce random knot diagrams. Combining this chain with flat histogram methods we achieve an efficient sampler of plane curves and knot diagrams. By analysing data from this chain we are able to estimate the number of knot diagrams of a given size and also compute knotting probabilities and so investigate their asymptotic behaviour.
\end{abstract}
\maketitle

\section{Introduction}
\label{sec:introduction}

\subsection{Background}
\label{sec:background}

A beautiful aspect of knot theory is that it brings together many different aspects of mathematics and leverages their tools so efficiently. An apt example of this phenomenon are knot diagrams, wherein the entanglement of a smooth string in space is studied as a combinatorial cartoon. From this view, any of the infinite manipulations of a loop of string are simply sequences of a finite number of diagram operations called Reidemeister moves~\parencite{Alexander_1926,Reidemeister1948}. This combinatorial regime also serves an essential role in knot identification; many knot polynomials are calculated more naturally for diagrams than for space curves~\parencite{Ewing1997}. Other invariants, like the important crossing number of a knot, are defined in terms of diagrams.

\begin{figure}[hbtp!]
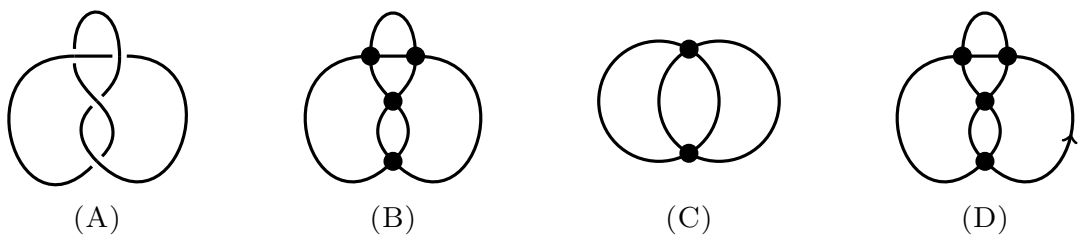

  \centering
  \begin{subfigure}{0.23\linewidth}
    \centering
    \includestandalone{figure_8_big}
    \caption{}
    \label{fig:knotdia}
  \end{subfigure}
  \begin{subfigure}{0.23\linewidth}
    \centering
    \includestandalone{figure_8_big_ci}
    \caption{}
  \end{subfigure}
  \begin{subfigure}{0.23\linewidth}
    \centering
    \includestandalone{hopf_ci}
    \caption{}
  \end{subfigure}
  \begin{subfigure}{0.23\linewidth}
    \centering
    \includestandalone{figure_8_big_ci_root}
    \caption{}
  \end{subfigure}
  \caption{By replacing crossings with vertices, knot diagrams (A) are mapped onto plane curves (B). However, not all 4-valent plane graphs correspond to knots. For example, the Hopf-link (C) maps to a 4-valent graph. To avoid complications caused by symmetries, we study rooted plane curves, in which one edge is selected and assigned an orientation (D).}
  \label{fig:knotty things}
\end{figure}

Knot theory is an important tool in many applied disciplines. Ring polymers exhibit knotting~\parencite{Frisch1961,Trigueros2001}, which affects their function~\parencite{Buck04} and their chemical properties~\parencite{Vanderzande1995}. Knotted configurations of polymers are studied through \emph{random knot models}, where knots are sampled from some probability distribution with the aim of modeling physical behavior; for a good overview see~\parencite{Orlandini07}. A canonical model of random knotting is that of self-avoiding polygons on the cubic lattice \(\Z^3\)~\parencite{Sumners_1988}. A self-avoiding polygon is constructed by embedding a simple closed curve into \(\Z^3\) and by sampling these objects uniformly at random, one obtains a distribution over the space of knot embeddings. A surprising amount is still unknown for all but the most simple of such models. For example, it is still unproven (despite overwhelming evidence) that the exponential growth rates of knotted polygons of fixed knot type are \emph{independent} of knot type~\parencite{Rensburg2011}.

As knot diagrams are more naturally suited for the study of invariants, in~\parencite{Cantarella2015}, the first author together with Cantarella and Mastin proposed the \emph{random diagram model}. In this model, one samples a \emph{random knot} by, for a fixed number of crossings \(n\), picking one of the finite number of knot diagrams with \(n\) crossings uniformly. The diagram model behaves like many other physically motivated models in that unkotted diagrams are exponentially rare~\parencite{Frisch1961,Delbruck1962,Sumners_1988,Diao1995,Chapman2016}. At the same time, this diagram model has advantages over other models of knotting. It is possible, for example, to show, via a pattern theorem, that unknot diagrams almost certainly contain slipknots~\parencite{Chapman2016:sub}, a result which remains a conjecture for unknotted self-avoiding polygons and Gaussian polygons~\parencite{Millett2010}.

Another major problem of combinatorial models of knots, is that the underlying objects are difficult to enumerate. The diagram model is no different --- the best algorithms for enumerating knot diagrams require time that grows exponentially with the number of crossings \parencite{ZinnJustin2009}. If an efficient enumeration method were to exist, then this could be readily adapted to give a random sampling method (see, for example, \parencite{Flajolet2009}). In the absence of such a method, the most obvious approach has been to generate random 4-valent maps and then randomly assign crossings (see Figure~\ref{fig:knotty things}). However, 4-valent maps corresponding to knots are exponentially rare. 
Experiments suggest that roughly only 1\% of samples are accepted for \(60\)-crossing diagrams and a prohibitively small 0.01\% of samples are accepted for \(150\)-crossing diagrams \parencite{Chapman2016}. This makes rejection sampling of 4-valent maps extremely inefficient at even moderate sizes. One way to overcome this is to manipulate a sampled 4-valent map until a knot diagram is obtained~\parencite{Diao2012,Dunfield2014talk}, however the resulting space of knots is not uniform. 

The absence of an efficient random knot diagram sampler is a major impediment to the study of random knots. The aim of this paper is thus to describe a new efficient method to sample knot diagrams \emph{directly} with uniform probability using Metropolis style Markov chain Monte Carlo (MCMC) sampling~\parencite{Metropolis_1953}. Such Markov chains have been used to study other models of random knotting, particularly on the simple cubic lattice. Foremost among these are the pivot algorithm~\parencite{lal1969} and the BFACF algorithm~\parencite{Berg1981,deCarvalho1983,deCarvalho1983_2}. The former is extremely efficient~\parencite{madras1988}, while the latter has the advantage of conserving topology~\parencite{Rensburg91}. Guitter and Orlandini~\parencite{Guitter99} augmented the BFACF algorithm with Reidemeister moves to study a model of flat-knots on the square-diagonal lattice.

This paper focuses on sampling \emph{plane curves}, which can be thought of as knot diagrams without crossing-sign information (see Figure~\ref{fig:knotty things} and Section~\ref{sec:prelims}). Each plane curve of \(n\)-vertices maps to a unique set of \(2^n\) knot diagrams. Consequently if we can sample plane curves uniformly, then we can also sample knot diagrams uniformly. To avoid complications caused by symmetries, we sample from the space of rooted diagrams (explained below). 

The main result of the paper is a Markov Chain over the space of plane curves. By selecting transition probabilities we can sample from this chain with a Boltzmann distribution.
\begin{theorem}\label{thm:ergodic}
  Let $D$ be a plane curve with \(1 \leq n\) vertices and let $\mu$ is the exponential growth rate of plane curves. The Markov chain described in Section~\ref{sec:boltzmann} has stationary distribution given by 
  \[ \Prb(D) \propto z^{n}, \]
  provided $0\leq z < \mu^{-1}$. Consequently, plane curves of a fixed size are sampled uniformly.
\end{theorem}

When running this Markov chain, we found that it was very difficult to obtain a good number of samples at large range of lengths. To overcome this problem we modified our transition probabilities based on Wang Landau density of states estimation~\parencite{Wang01}. This method allows us to sample diagrams nearly uniformly in length (while still sampling uniformly within any given length) and additionally provides estimates of the number of plane diagrams. Let $k_\ell$ be the number of rooted plane curves with $\ell$ vertices, and let $g_\ell$ be the estimate of $k_\ell$ from the Wang-Landau algorithm. Then we have the following result.
\begin{theorem}\label{thm:wlergodic}
  Let $N \in \N$ and $D$ be a plane curve with \(1 \leq n \leq L\) vertices. The Markov chain described in Section~\ref{sec:wanglandau} has stationary distribution given by 
  \[ \Prb(D) \propto \frac{1}{g_{n}} \]
Since $g_n \approx k_n$, plane curves are sampled uniformly within a given size, and approximately uniformly across sizes.
\end{theorem}

The remainder of Section~\ref{sec:introduction} defines key knot theory concepts and provides some additional context on the problem of enumerating plane curves. Section~\ref{sec:transitions} describes the transitions used by our Markov chain. In Section~\ref{sec:markovchain}, we describe the actual algorithms and prove the main theorem. Section~\ref{sec:data} presents the results of experiments which verify the validity of the main theorem as well as explore the structure of large random knot diagrams. Finally, in the concluding Section~\ref{sec:conclusion} we discuss progress on additional ``diagram Markov chains'' for different types of diagram objects, and some obstructions.

\section{Preliminaries and Definitions}
\label{sec:prelims}

\subsection{Definitions}
\label{sec:definitions}

A \emph{knot} is an embedding \(K: S^1 \hookrightarrow \R^3\) of a loop into Euclidean 3-space. Typically, knots are considered up to ambient isotopy, wherein two knots are equivalent if one can be manipulated as a closed loop into the other, without self-intersection. For clarity, we call a specific loop embedding a \emph{knot} and an equivalence class of knots a \emph{knot type}. Reidemeister's theorem~\parencite{Alexander_1926,Reidemeister1948} transfers this topological theory into a combinatorial one as follows:
\begin{figure}[hbtp! b]
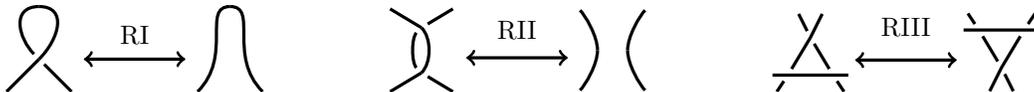

  \centering
    \begin{tabular}{c@{\hspace{4em}}c@{\hspace{4em}}c}
    \centering
    \includestandalone{reidemeister_1}
    & \includestandalone{reidemeister_2}
    & \includestandalone{reidemeister_3}
  \end{tabular}
  \caption{The three Reidemeister moves, RI, RII, RIII.}
  \label{fig:reidemeister}
\end{figure}
A \emph{knot diagram} of a knot \(K\) is a generic projection of the loop in space to the sphere, together with extra information at each double point indicating where one piece of the loop passes over the other (called \emph{crossings}), as in Figure~\ref{fig:knotdia}, up to oriented homeomorphisms of the sphere. Then two knots \(K_1\) and \(K_2\) are equivalent if and only if their \emph{diagrams} are related by a sequence of Reidemeister moves, depicted in Figure~\ref{fig:reidemeister}. A significant advantage of the diagram view is that most \emph{knot invariants}, properties of knots which only depend on their knot type, are naturally computed from their diagram representation~\parencite{Kauffman87,Freyd85}.




There is a natural projection from knot diagrams to a strict subset of 4-valent planar maps called plane curves; simply replace each crossing with a vertex. A \emph{planar map} is a (multi-)graph \(G\), together with an embedding \(\iota\) into the sphere \(S^2\) so that each component of \(S^2 \setminus \iota(G)\) is a topological disk. Necessarily, this means all planar maps are connected. A map is \emph{4-valent} if each vertex has degree 4.

Unfortunately, knot diagrams and plane curves are cumbersome to deal with as a result of potential symmetries. To avoid these complications can asymmetrize by marking one edge with a direction. Such objects are called \emph{rooted}. In particular, each rooted \(n\)-vertex plane curve corresponds to a unique set of exactly \(2^n\) rooted knot diagrams. For a discussion of the techniques and difficulties involved with considering plane curves and knot diagrams with symmetry see~\parencite{Coquereaux16,Cantarella2015,Valette2016}.

It will be useful for computations to consider the following equivalent view of maps. A 4-valent planar map \(D\) with \(n\) vertices can be viewed as a \emph{combinatorial map}~\parencite{Coquereaux16,Chapman2016}, \textit{i.e.} a pair \(D = (\sigma, \tau)\) of permutations of \(4n\) \emph{flags} (sometimes called half-edges or arcs). In this view, \(\sigma\) is a product of \(n\) disjoint cycles of length 4 and \(\tau\) is a product of \(2n\) disjoint cycles of length 2. Each cycle in \(\sigma\) represents a vertex (it permutes the flags attached at each vertex counterclockwise) and each cycle in \(\tau\) represents an edge (it involutes the two flags that form an edge). The cycles of \(\sigma \tau\) correspond to the faces of the map (each permutes the flags of a face clockwise). For 4-valent maps we also have \(\sigma^2 \tau\), whose cycles correspond to orientations of \emph{link components} or Gauss components. Each flag is contained in exactly one vertex, edge, face and component.

Given an flag \(a\), let \(e(a)\) be its edge (\textit{i.e.} cycle in \(\tau\)), \(v(a)\) be its vertex (\textit{i.e.} cycle in \(\sigma\)), and \(f(a)\) be its face (\textit{i.e.} cycle in \(\sigma\tau\)).

\begin{figure}[hbtp!]
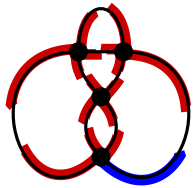

  \centering
  \includestandalone{figure_8_big_ci_arcs}
  \caption{A plane curve, with its flags marked and a root flag chosen (in blue).}
  \label{fig:circleimmex}
\end{figure}

Then the condition that \(D\) is a planar map is precisely that product \(\sigma\tau\) consists of \(n+2\) cycles by Euler's formula. Relaxing this condition and allowing \(\sigma\tau\) to consist of \(k\) cycles makes \(D\) a map on a surface of genus \(g = 1-\frac{k-n}{2}\). We forbid this for our objects, although in general it is interesting to consider maps on an arbitrary fixed surface \(\Sigma\). If \(\sigma^2\tau\) consists of precisely two cycles, each necessarily of length \(n\), then \(D\) is a \emph{plane curve} and each cycle in \(\sigma^2\tau\) corresponds to following the single immersed circle in one of its two possible orientations. Figure~\ref{fig:circleimmex} shows an example plane curve. Figure~\ref{fig:randomdia} shows a random decorated plane curve (knot diagram) and a random decorated 4-valent map (link diagram). The curve condition (that \(\sigma^2 \tau\) has precisely 2 cycles) makes plane curves exponentially rare within the class of all 4-valent maps.

We note that if we relax either the planarity or the curve condition, the problem greatly simplifies. 4-valent planar maps themselves are well-understood, owing in part to Schaeffer's bijection with blossom trees~\parencite{Schaeffer1997}. This has been used to prove a stunning closed formula for the growth rate of alternating link types~\parencite{Sundberg_1998,Thistlethwaite_1998,Zinn_Justin_2002} as well as precise statistics for hyperbolic volumes of random alternating link diagrams~\parencite{Obeidin16}. On the other hand, relaxing planarity and considering curves on arbitrary surfaces leads one to the study of \emph{Gauss codes} (usually depicted as signed chord or \emph{Gauss diagrams}), which themselves are counted and well understood~\parencite{Nowik2009}. 

Denote by \(\KnotShad\) the class of all rooted plane curves indexed by number of vertices, and let \(K(t) = \sum_{n=1}^\infty{k_nt^n}\) be its generating function. For a plane curve \(D\), let \(|D|\) be the number of vertices in \(D\) (equivalently, the size of \(D\)). Then the generating function can also be written as \(K(t) = \sum_{D\in{\KnotShad}}{t^{|D|}}\). The asymptotic behavior of the coefficients \(k_n\) is expected~\parencite{Schaeffer2004,ZinnJustin2009} to be
\begin{equation}
  \label{eqn:asympkn}
 k_n \sim C\mu^nn^{\gamma-2}(1 + O(1/\log n)).
\end{equation}

Neither a closed formula for \(k_n\) nor exact values of \(\mu, \gamma\) are known. Conformal field theory arguments suggest~\parencite{Schaeffer2004} that
\[ \gamma = -\frac{1 + \sqrt {13}}{6}. \] Additionally, it is known that \(\mu\) exists~\parencite{Chapman2016}, with the best numerical estimate~\parencite{ZinnJustin2009} \(\mu \approx 11.416 \pm 0.005\). This sort of asymptotic growth is similar to that of self-avoiding walks and polygons in the cubic lattice \(\Z^3\)~\parencite{Hammersley1961,Madras2013} (with different constants). Indeed a great many combinatorial objects are known to be counted by sequences which have similar exponential growth with power-law correction; see Flajolet and Sedgewick~\parencite{Flajolet2009} for (a great many) examples. It should be noted, however, that conformal field theory arguments~\parencite{Schaeffer2004} suggest the presence of an inverse logarithmic correction to scaling for plane curves. This is contrast with the observed correction of $n^{-\Delta}$ for many objects (again, see Flajolet and Sedgewick~\parencite{Flajolet2009} for many examples with $n^{-1}$ corrections, and also~\parencite{Conway1996} for evidence of a $n^{-3/2}$ correction in self-avoiding walks).

\begin{figure}[hbtp!]
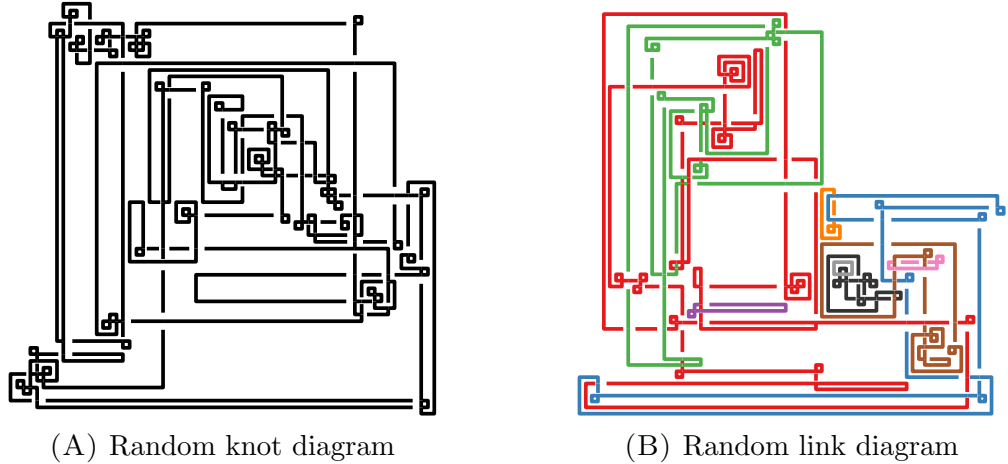

  \centering
  \begin{subfigure}[t]{0.45\textwidth}
    \centering
    \includestandalone{random_100kd}
    \caption{Random knot diagram}
    \label{subfig:kd}
  \end{subfigure}
  \begin{subfigure}[t]{0.45\textwidth}
    \centering
    \includestandalone{random_100ld}
    \caption{Random link diagram}
    \label{subfig:ld}
  \end{subfigure}
  \caption{A random knot diagram and a random link diagram, each of 100 vertices. Different link components are given different colors. These diagrams were sampled uniformly (using a rejection sampler in the case of the knot diagram) using an interface in \texttt{plCurve}~\parencite{PlCurve} to \texttt{PlanarMap}~\parencite{SchaefferPlanarMap}, and graphics were generated using an orthogonal projection algorithm in \texttt{pLink}, part of \texttt{SnapPy}~\parencite{SnapPy}. Knot diagrams become exponentially rare as the number of vertices increases~\parencite{Schaeffer2004,Chapman2016}, so are difficult to sample through rejection.}
  \label{fig:randomdia}
\end{figure}

\subsection{Shadow Reidemeister moves}
\label{sec:transitions}

In the subsection above, we have defined the set of rooted plane curves that we wish to sample. Unfortunately, as noted above, it is difficult to construct a rooted plane curve. Instead, we will describe a Markov chain that produces new plane curves by performing small local changes. The set of plane curves is closed under these manipulations. Further, any two plane curves are linked by a sequence of these changes.

\begin{figure}[hbtp! b]
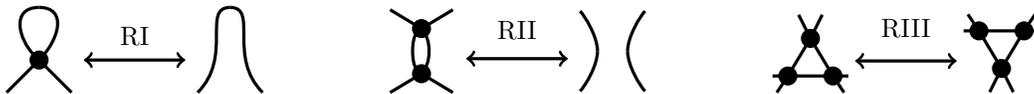

  \centering
    \begin{tabular}{c@{\hspace{4em}}c@{\hspace{4em}}c}
    \centering
    \includestandalone{flat_reidemeister_1}
    & \includestandalone{flat_reidemeister_2}
    & \includestandalone{flat_reidemeister_3}
  \end{tabular}
  \caption{The three flat Reidemeister moves, RI, RII, RIII which act on plane curves. These are the natural analogues of Reidemeister moves on knot diagrams.}
  \label{fig:flat_reidemeister}
\end{figure}


Reidemeister moves in Figure~\ref{fig:reidemeister} are an obvious choice of local changes for knot diagrams, and we use a similar set of moves for plane curves. Define the flat (or shadow) Reidemeister moves (also known as homotopy moves~\parencite{Chang_2017}) to be the same as the Reidemeister moves, except ignoring crossing information~\parencite{Henrich2010}. Any two plane curves are related by a sequence of flat Reidemeister moves, as;

\begin{theorem}[Hass and Scott~\parencite{Hass_1994}, de Graaf and Schrijver~\parencite{de_Graaf_1997}]
  \label{thm:trivplanecurve}
  Any plane curve can be brought to the trivial figure-eight twist curve seen in Figure~\ref{fig:figeight} by a sequence of flat Reidemeister moves that never increase the size of the plane curve.
\end{theorem}

\begin{figure}
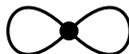

  \centering
  \includestandalone{fig8_twist}
  \caption{The figure-eight twist curve is the ``trivial'' plane curve for the sake of implementation; allowing for a trivial curve of 0 vertices is more difficult and provides no benefit. By Theorem~\ref{thm:trivplanecurve} all other plane curves can be brought to this shape by a sequence of non-crossing-increasing flat Reidemeister moves.}
  \label{fig:figeight}
\end{figure}

This result implies that there exists a sequence of moves between any two plane curves \(D, N\). Furthermore, it implies that at each intermediate state between \(D\) and \(N\) the curve has no more vertices than the larger of \(D\) and \(N\). Chang and Erickson have proven in the arbitrary case that the maximum number of moves required is a small polynomial:
\begin{theorem}[\cite{Chang_2017}]
  The maximum number of non-increasing flat Reidemeister moves required to trivialize a plane curve with \(n\) vertices grows as \(\Theta(n^{3/2})\). Consequently, it takes no more than \(\Theta(n^{3/2})\) moves to transform a curve of \(m \le n\) vertices into a curve of \(n\) vertices.
\end{theorem}
A similar pair of results, if flat Reidemeister I moves are forbidden, was proven by Nowik~\parencite{Nowik2009}. In this case, the path of curves between a curve with \(n\) vertices and its trivialization takes no more than \(\Theta(n^2)\) moves and involves no intermediate curves with more than \(n+2\) vertices.

These flat Reidemeister moves form the transitions for our Markov chain on \(\KnotShad\). A key property is that they are all reversible. We define each transition in detail, noting their unique inverses, since such details are necessary for efficient implementation. Let \(D\) be a rooted plane curve with root flag \(a\).

\begin{figure}[hbtp!]
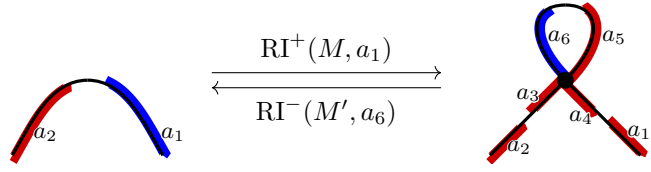

  \centering
  \includestandalone{flatr1}
  \caption{The flat Reidemeister I moves}
  \label{fig:flatr1}
\end{figure}


\subsubsection{Shadow Reidemeister I loop addition, \(\RIp\)} See Figure~\ref{fig:flatr1} (left) and consider the edge to be oriented from right to left. This defines two flags, \(a_1\) and \(a_2\), and we take \(a_1\) to be the root flag. The addition of a loop on \( a_1\) is always possible through the move \(\RIp\).

Let \((a_1a_2) = e(a_1)\), and let \(a_3,a_4,a_5,a_6\) be four new flags. Then \(\RIp(D,a_1)\) is the rooted map produced by deleting edge \(e(a_1)\) from \(D\), then adding the vertex \((a_3a_4a_5a_6)\) and edges \((a_1a_4)\), \((a_5a_6)\), and \((a_3a_2)\). The new root is the flag \(a_6\). One can verify that this process is invertible, in particular that:
\[D = \RIm(\RIp(D,a_1),a_6).\]

\subsubsection{Shadow Reidemeister I loop deletion, \(\RIm\)} See Figure~\ref{fig:flatr1} (right). Loop deletion
\(\RIm\) is possible whenever, given the root flag \(a = a_6\), the face \(f(a_6)\) is a singleton (\textit{i.e.} the corresponding face is a monogon).

If this is true, then \(\RIm(D,a_6)\) is the rooted map produced from deleting the vertex \(v(a_6) = (a_3a_4a_5a_6)\) and the edges \(e(a_6) = (a_6a_5)\), \(e(a_3) = (a_3a_2)\), \(e(a_4) = (a_4a_1)\) from \(D\), and adding the edge \((a_1a_2)\). The flags \(a_3, a_4, a_5, a_6\) are discarded. The new root is the flag \(a_1\). This process is invertible: \[D = \RIp(\RIm(D,a_6),a_1).\]

\begin{figure}[hbtp!]
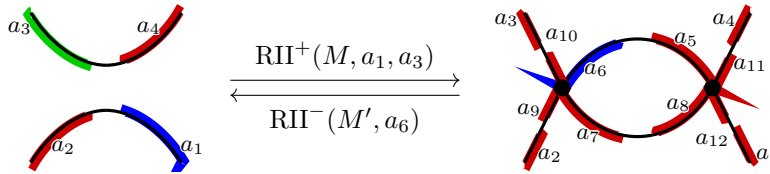

  \centering
  \includestandalone{flatr2}
  \caption{The flat Reidemeister II moves}
  \label{fig:flatr2}
\end{figure}

\subsubsection{Shadow Reidemeister II bigon addition, \(\RIIp\)} See Figure~\ref{fig:flatr2} (left). Bigon addition
\(\RIIp\) requires a second flag \(b \in f(a) \)  which is different from the root flag \(a\). Let \(a = a_1\) and \(b = a_3\).

The rooted map \(\RIIp(D,a_1,a_3)\) is constructed from \(D\) as follows. Delete
the edges \(e(a_1) = (a_1a_2)\) and \(e(a_3) = (a_3a_4)\) from \(D\). Add eight
new flags \(a_5,a_6,a_7,a_8,a_9,a_{10},a_{11},a_{12}\). Add the vertices
\((a_6a_{10}a_9a_7)\) and \((a_5a_8a_{12}a_{11})\). Insert the edges
\((a_2a_9)\), \((a_3a_{10})\), \((a_7a_8)\), \((a_5a_6)\), \((a_1a_{12})\), and
\((a_4a_{11})\). The new root is the flag \(a_6\). This process is invertible:
\[D = \RIIm(\RIIp(D,a_1,a_3), a_6). \]

\subsubsection{Shadow Reidemeister II bigon deletion, \(\RIIm\)} See Figure~\ref{fig:flatr2} (right). Bigon deletion \(\RIIm\) is possible provided the root flag \(a = a_6\) is on a face \(f(a_6)\) which consists of precisely two flags (\textit{i.e.} a bigon).

Additionally, it is required that the two exterior faces which are merged by the
transition be \emph{distinct}; this is required to preserve connectedness and
genus. Indeed, Suppose that \(D\) is a curve embedded on an orientable surface
of genus \(g\), so that \(v(D) - e(D) + f(D) = 2(g-1)\), but that the faces to
be merged are not distinct. The number of vertices and edges decrease by 2 and 4
respectively by a \(\RIIm\) operation, as usual, but the number of faces now
\emph{remains fixed}. This implies that the produced curve lives in either a
surface of one fewer genus, or if the original map was embedded on the sphere,
\emph{two disjoint spheres}.

The rooted map \(\RIIm(D,a_6)\) is constructed from \(D\) as follows. Delete the
vertex \(v(a_6) = (a_6a_{10}a_9a_7)\) and the edges \(e(a_9) = (a_2a_9)\),
\(e(a_{10}) = (a_3a_{10})\), \(e(a_7) = (a_7a_8)\) and \(e(a_6) = (a_6a_5)\).
Delete the vertex \(v(a_5) = (a_5a_8a_{12}a_{11})\) and the edges \(e(a_{12}) =
(a_1a_{12})\) and \(e(a_{11}) = (a_4a_{11})\). Add the edges \((a_1a_2)\) and
\((a_3a_4)\). The flags \(a_5,a_6,a_7,a_8,a_9,a_{10},a_{11},a_{12}\) are
discarded. The new root is the flag \(a_1\). This process is invertible:
\[D = \RIIp(\RIIm(D,a_6),a_1,a_3). \]

\begin{figure}[hbtp!]
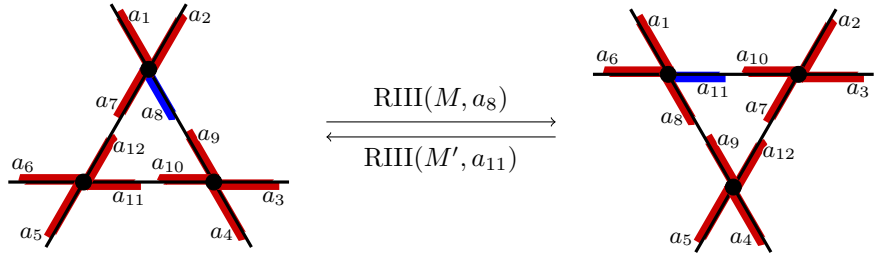

  \centering
  \includestandalone{flatr3}
  \caption{The flat Reidemeister III move}
  \label{fig:flatr3}
\end{figure}

\subsubsection{Shadow Reidemeister III triangle flip, \(\RIII\)} See Figure~\ref{fig:flatr3}. Triangle
flipping \(\RIII\) is possible provided the root flag \(a = a_8\) lies on a
face \(f(a_8)\) with precisely three flags, all of whom are contained
in \emph{different} vertices (\textit{i.e.} a nondegenerate triangle).

The rooted map \(\RIII(D,a_8)\) is constructed from \(D\) as follows. Say that
\(v(a_8) = (a_1a_2a_7a_8)\), \(e(a_8) = (a_8a_9)\), \(e(a_7) = (a_6a_7)\),
\(v(a_6) = (a_5a_{11}a_{12}a_6)\), and \(v(a_9) = (a_3a_9a_{10}a_4)\). Delete
vertices \(v(a_7)\), \(v(a_6)\), \(v(a_9)\) and edges \(e(a_7)\), \(e(a_8)\),
and \(e(a_{10})\). Insert vertices \((a_6a_8a_{11}a_1)\), \((a_2a_{10}a_7a_3)\),
\((a_4a_{12}a_9a_5)\) and edges \((a_6a_9)\), \((a_7a_{12})\), and
\((a_{10}a_{11})\). The new root is the flag \(a_{11}\). This process is invertible:
\[ D = \RIII(\RIII(D,a_8),a_{11}).\]

\section{Markov Chain}
\label{sec:markovchain}
The five flat Reidemeister moves described in the previous section allow us to define a Markov chain on the space plane curves. That the chain can move between any two given plane curves follows immediately from Theorem~\ref{thm:trivplanecurve}.

\subsection{A Boltzmann Markov chain on plane curves}
\label{sec:boltzmann}

Our Markov chain sampler for plane curves will have a stationary \emph{Boltzmann distribution}; one which samples curves of different sizes with different probabilities but is \emph{uniform on curves of fixed size}. In other words, for \(0 \le z < \mu^{-1}\) an arbitrary parameter, it has stationary distribution such that a curve \(D\) has probability
\begin{equation}
\Prb(D) \propto z^{|D|}
\end{equation}
The parameter \(z\) then is then maybe used to control mean size of sampled curves, and smaller values of \(z\) will prevent the samples from growing infinitely large.

Let \(p_1, p_2, p_3 > 0\). These numbers correspond to the probabilities of performing, respectively, a Reidemeister I, II, or III move. Consequently, we must have \(p_i > 0\). Further, we must also have \(1-(p_1+p_2+p_3) > 0\), as this quantity will correspond to the probability of selecting a new root flag. Let \(D_i\) be the input rooted plane curve with \(n\) vertices and root flag \(a\) and perform one of the following six subprocedures with different probabilities.  If a move fails then set \(D_{i+1} := D_i\).

\begin{itemize}
\item[ {\( [0] \)}] Re-rooting, with probability \(1-(p_1+p_2+p_3)\). Given the rooted diagram \(D_i\), forget the root and select a new root \(b\) for \(D_{i+1}\) from the \((\Aut{D_{i}})/(4|D|)\) choices. The probability that this transition succeeds is \(1-(p_1+p_2+p_3)\). Note that \(\Aut{D_{i}}\) needn't be calculated for this operation, as it is equivalent to choose one flag from the old rooted curve.
  
\item[{\([1^+]\)}] Loop addition, with probability \(p_1/2\). Sample \(0 \le \alpha < 1\) and fail immediately if \(\alpha > z\). Return \(D_{i+1} := \RIp(D_i,a)\). The probability that this transition succeeds on an \(n\)-crossing plane curve \(D\) is \(zp_1/2\).
  
\item[{\([1^-]\)}] Loop deletion, with probability \(p_1/2\). Provided \(f(a)\) is a loop, return \(D_{i+1} := \RIm(D_i,a)\). The probability that this transition succeeds if root flag \(a\) has \(f(a)\) a loop in \(D\) is \(p_1/2\).
  
\item[{\([2^+]\)}] Bigon addition, with probability \(p_2/2\). Sample \(0 \le \alpha < 1\) and fail immediately if \(\alpha > z^2\). The flag \(a\) lies along a face of \(d\) edges; provided \(d \ne 1\) (otherwise fail), uniformly sample the integer \(k\) between \(1\) and \(d-1\). The flag \(a' = (\sigma\tau)^k(a)\) is a distinct flag along the same face as \(a\). Then return \(D_{i+1}:=\RIIp(D_i,a,a')\). The probability that this transition succeeds on any given additional flag \(a'\) along the root \(d\)-face is
  \[ \frac {z^2p_2}{2(d-1)}. \]
  
\item[{\([2^-]\)}] Bigon deletion, with probability \(p_2/2\). Fail if the flag \(a\) does not lie along a bigon. The size \(d\) of the face which would be produced by bigon deletion is the sum \(|f(\tau\sigma(a))|+|f(\sigma^3\tau(a))| - 2 = d\). Sample \(0 \le \beta < 1\) uniformly and fail if \(\beta > (d-1)^{-1}\). Otherwise, return \(D_{i+1} := \RIIm(D_i,a)\). The probability that this transition succeeds on a root \(a\) along a bigon is
  \[ \frac {p_2}{2(d-1)}. \]
  
\item[{\([3]\)}] Triangle flipping, with probability \(p_3\). Fail if the flag \(a\) does not lie along a nondegenerate triangle. Otherwise, return \(D_{i+1} := \RIII(D_i,a)\). The probability that this transition succeeds assuming the root lies along a nondegenerate triangle is \(p_3\).
  
\end{itemize}

We will prove that in the limiting distribution the probability that any given \(n\)-crossing rooted plane curve \(D\) is chosen is
\begin{equation}
\Prb(D) = \frac{z^n}{K(z)},
\end{equation}
where \(K(z)\) is the value of the generating function \(K(t)\) at \(z\) so if \(K(z)\) converges to a number, \(\Prb(D) \propto z^n\) (this happens provided \(z < \mu^{-1}\)). It follows that, if we ignore the roots of the sampled diagrams in order to sample \emph{unrooted} diagrams, the probability of an unrooted diagram \(\overline D\) being sampled is
\begin{equation}
 \Prb(\overline D) = \frac{4n}{\Aut{\overline D}} z^n.
\end{equation}
We note that as the probability that the automorphism group of a plane curve is trivial tends exponentially quickly to 1~\parencite{Chapman2016}, the typical probability of an unrooted diagram \(\overline D\) will be \(\Prb(\overline D) = 4nz^n\).

This Markov chain is ergodic as it satisfies the following three properties;
\begin{enumerate}
\item It is connected: It is possible to get from any one plane curve to another in a finite number of transitions. Provided \(p_1, p_2, p_3 > 0\) it is possible to get between any two pairs of unrooted curves (by Theorem~\ref{thm:trivplanecurve}); provided \(p_1 + p_2 + p_3 < 1\) it is guaranteed that any flag may be chosen as the root.
\item It is aperiodic: Since at each step there is a non-zero probability that the transition failes, there is no periodicity in the Markov chain.
\item The chain satisfies detailed balance: For any two curves \(D\) and \(N\), the transition probabilities \(\Trprb\) and curve probabilities \(\Prb\) satisfy,
  \[ \Trprb(D \to N)\Prb(D) = \Trprb(N \to D)\Prb(N). \]
  This last point requires the most care and we discuss it below, with more details in Appendix~\ref{sec:detailedbalanceproofs}.
\end{enumerate}
The fundamental theorem of Markov chains then yeilds the following result.

\begin{reptheorem}{thm:ergodic}[Slightly restated]
This Markov chain satisfies detailed balance. Furthermore, if \(p_1, p_2, p_3 \ne 0\), \((p_1+p_2+p_3) \ne 1\), and \(z < \mu^{-1}\) the chain is ergodic with stationary distribution,
  \[ \Prb(D) = \frac{z^{|D|}}{\sum_{\ell}k_\ell z^\ell} \propto z^{|D|}. \]
\end{reptheorem}

The proof of this result is primarily routine. We have omitted some details which can be found in Appendix~\ref{sec:detailedbalanceproofs}. 

\begin{proof}
  We will begin by assuming that
  \begin{align}
  \Prb(D) = \frac{z^{|D|}}{\sum_{\ell}k_\ell z^\ell} \propto z^{|D|}, 
  \end{align}
  and proving that detailed balance holds with this hypothesis. Notice that \(z\) has been chosen sufficiently small so that the denominator converges~\parencite{Chapman2016}. In all cases the denominator is the same and a common factor in the calculations that follow, so we omit it.
  
  Let \(D\) be a rooted plane curve of \(n\) vertices and \(a\) be the root flag in \(D\). Observe first that the three pairs of reversing transitions \((\RIp, \RIm)\), \((\RIIp, \RIIm)\), and \((\RIII, \RIII)\) all change the number of vertices by distinct complementary amounts; hence any two diagrams can be related by at most one pair of these transitions. Notice that if the vertex counts agree, then a pair of diagrams may be related by re-rooting. The main concern now is to prove the detailed balance equations for all possible transitions. Full details are in Appendix~\ref{sec:detailedbalanceproofs}; we demonstrate that detailed balance holds in the case of the \(\RII\) moves, as an example.

   Suppose that \(N = \RIIp(D,a,a')\) with root flag \(b\). This means that \(N\) is unique in that \(D = \RIIm(N,b)\). The flags \(a, a'\) lie along a face in \(D\) of degree \(d\). Then
  \begin{align}
    \Trprb(D \to N)\Prb(D) &= \Trprb(N \to D)\Prb(N) \\
    \frac{z^2p_2}{2(d-1)}z^n &= \frac{p_2}{2(d-1)}z^{n+2}.
  \end{align}
  In all other cases, the transition probabilities are symmetrically zero. Hence we conclude that detailed balance holds with the hypothesized probability distribution.

  Provided \(p_1, p_2, p_3 \ne 0\) and \((p_1+p_2+p_3) \ne 1\), the Markov chain can reach all rooted plane curves as all flat Reidemeister moves and the rerooting move have nonzero transition probability. Hence in this case, the Markov chain is ergodic.
\end{proof}

Notice that, by Theorem~\ref{thm:trivplanecurve}, any plane curve can be reduced to the trivial plane curve by a sequence of flat Reidemeister moves which \emph{never} increases the number of vertices. 
This property implies that if we restrict our Markov chain to the set of plane curves of size at most \(L\), then the chain remains ergodic.  This is definitely not true for \emph{knot diagrams} \parencite{Kauffman2006Hard}. There exist diagrams of knots which are locally minimal in the sense that they can be reduced in crossing number, but only through sequences of moves which increase the crossing number at some point. 

\begin{corollary}
  Imposing the further restriction on the \emph{loop addition} and \emph{bigon addition} transitions that we fail if the input plane curve \(D\) would transition to have more than \(L\) vertices (\emph{i.e.}\ if \(D\) has \(L\) or \(L-1\) vertices, respectively) yields a Markov chain which explores all of \(\KnotShad_{i\le L}\) and is ergodic. Furthermore, the probability of sampling any plane curve \(D\) with \(m \le L\) crossings is,
  \[ \Prb(D) = \frac{z^m}{\sum_{\ell=1}^{L}{k_{\ell}z^\ell}} \propto z^m. \]
  \label{cor:capncross}
\end{corollary}

\begin{proof}
  This follows from Theorem~\ref{thm:trivplanecurve} and the proof of the previous theorem.
\end{proof}

At this stage we are free to choose \(p_1, p_2, p_3\). It is not obvious, how they should be chosen to produce the most efficient sampling method. In practice, we simplified the problem by choosing $p_1+p_2+p_3=1$ and then applying a rerooting move after each step. The ergodicity of this chain follows by very similar arguments.
\begin{corollary}
  For \(p = (p_1, p_2, p_3)\) with \(p_1 + p_2 + p_3 = 1\), consider the Markov chain which differs from the prior by, before and after each transition step, uniformly randomly re-rooting the state curve. Setting \(p_1 + p_2 + p_3=1\) eliminates the re-rooting transition. Then this Markov chain is ergodic.
  \label{cor:noreroot}
\end{corollary}

The proof is quite standard, we give it in Appendix~\ref{sec:detailedbalanceproofs}. Notice that in practice, there is no need between consecutive transitions to re-root more than once. In our experiments, we choose the five flat Reidemeister moves by taking \(p_1 = 2/5\), \(p_2 = 2/5\) and  \(p_3 = 1/5\).

\subsection{Flat histogram sampling by Wang-Landau state density estimation}
\label{sec:wanglandau}

Our simulations in Section~\ref{sec:bzparams} show that the output curve sizes have high variance and change rather drastically with the parameter \(z\). There are several different approaches that one might use to overcome this --- such as multiple Markov chain Monte Carlo \parencite{geyer1991, orlandini1998mmcmc}. We have, instead, chosen the flat histogram method invented by Wang and Landau~\parencite{Wang01}. The transition probabilities are chosen so that the number of objects sampled at size \(n\) is approximately equal to the number of objects sampled at size \(m\). This requires that 
\begin{align}
  \frac{\Prb(D)}{\Prb(N)} & \approx \frac{k_{|N|}}{k_{|D|}}
\end{align}
Wang and Wang and Landau's algorithm~\parencite{Wang01} achieves this by estimating the density of states --- \textit{i.e.} by estimating \(k_n\). Consequently we must fix a maximum size $L$ and set the probabilities of transitions to diagrams with more than $L$ vertices to zero. The algorithm begins with a tuning phase which runs through the Markov chain adjusting transition properties until the above condition approximately holds. These final transition properties are then used for a MCMC sampler.

There are three advantages to this strategy over Boltzmann MCMC sampling:
\begin{enumerate}
\item Mixing of the Markov chain is more efficient: Rather than get caught up in one small range of sizes (in part due to the \emph{local minimum} bottleneck; see Section~\ref{sec:bzparams}), with Wang-Landau tuning the Markov chain moves frequently between all sizes. This improves the independence of two samples of the same size.

\item Guarantee of sampling objects of given size: The Markov chain produces objects of sizes up to $L$ with nonzero probability, so given enough tuning and enough random samples, a sufficient number of objects of a desired size will be sampled.

\item Approximate enumeration: Transition probabilities found during the tuning step are directly related to the counts of objects of each size. We use this data to estimate the number $k_n$ in Section~\ref{sec:wanglandaudata}.
\end{enumerate}

Rather than depending on a single parameter \(z\), sampling via a Wang-Landau implementation requires a data structure \((\ell,L,G)\), where:
\begin{itemize}
\item \(\ell\) is the minimum size of plane curves. To ensure ergodicity, we always have \(\ell = 1\), although further results on plane curves may allow \(\ell\) to vary (possibly in terms of \(L\)) while still guaranteeing ergodicity.
\item \(L\) is the maximum size of plane curves.
\item \(G = (G_i)_{i=\ell}^{L}\) is a vector approximate enumeration data in the following sense: If \(g_n = e^{G_n}\), then \(g_n/g_{n-1} \approx k_{n}/k_{n-1}\). Since we do not know \(k_n\) for \(n\geq28\)), this data is gathered via a tuning phase (described in Section~\ref{sec:wanglandau tuning} below).
\end{itemize}

Given the approximate enumeration data \(G_n = \log(g_n) \), define probabilities
\begin{align}
\tp(n, m) &= \min\left\{1, \exp(G_n - G_m)\right\} = \min\left\{1, g_n/g_m\right\} \approx k_n/k_m,
\end{align}
unless \(m < \ell\) or \(m > L\) in which case \(\tp(n,m) = 0\). Wang-Landau flattened MCMC sampling then works as follows. The implementation which follows is largely the same as the Boltzmann implementation, with each transition instead required to pass a check of probability \(\min\left\{1, g_n/g_{n+k}\right\}\), where \(n\) is the size of the input curve and \(k\) is the change in size of the transition operation. Let \(p_1, p_2, p_3 > 0\) and \((p_1+p_2+p_3) < 1\) and let \(D_i\) be the input rooted plane curve with \(n_i\) vertices and root flag \(a\) and perform one of the following six moves with different probabilities. If a move fails then set \(D_{i+1} := D_i\).

\begin{itemize}
\item[{\([0]\)}] Re-rooting, with probability \(1-(p_1+p_2+p_3)\). Given the rooted diagram \(D_i\), forget the root and select a new root \(b\) for \(D_{i+1}\) from the \((\Aut{D_{i}})/(4|D|)\) choices. The probability that this transition is chosen and succeeds is \(1-(p_1+p_2+p_3)\).

\item[{\([1^+]\)}] Loop addition, with probability \(p_1/2\). Sample \(0 \le \alpha < 1\) and fail immediately if \(\alpha > \tp(n_i, n_ {i+1})\). Return \(D_{i+1} := \RIp(D_i,a)\). The probability that this transition is chosen and succeeds on an \(n\)-crossing rooted plane curve is
  \begin{equation*}
    \tp(n, n+1) \frac{p_1}{2} = \min\left\{1,\frac{g_n}{g_{n+1}}\right\}\frac{p_1}{2}
  \end{equation*}
  
\item[{\([1^-]\)}] Loop deletion, with probability \(p_1/2\). Sample \(0 \le \alpha < 1\) and fail immediately if \(\alpha > \tp(n_i, n_{i-1})\). Provided \(f(a)\) is a loop, return \(D_{i+1} := \RIm(D_i,a)\). The probability that this transition is chosen and succeeds is \(f(a)\) is a loop is 
\begin{equation*}\tp(n, n-1) \frac{p_1}{2} = \min\left\{1,\frac{g_n}{g_{n-1}}\right\} \frac{p_1}{2} .\end{equation*}
  
\item[{\([2^+]\)}] Bigon addition, with probability \(p_2/2\). Sample \(0 \le \alpha < 1\) and fail immediately if \(\alpha > \tp(n_i, n_{i+2})\). The flag \(a\) lies along a face of \(d\) edges; provided \(d \ne 1\) (otherwise fail), uniformly sample the integer \(k\) between \(1\) and \(d-1\). The flag \(a' = (\sigma\tau)^k(a)\) is a distinct flag along the same face as \(a\). Then return \(D_{i+1}:=\RIIp(D_i,a,a')\). The probability that this transition is chosen and succeeds on an \(n\)-crossing plane curve with an given cofacial flag \(a'\) is
  \begin{equation*} \tp(n, n+2)\frac{p_2}{2(d-1)} = \min\left\{1,\frac{g_n}{g_{n+2}}\right\}\frac{p_2}{2(d-1)} \end{equation*}
  
\item[{\([2^-]\)}] Bigon deletion, with probability \(p_2/2\). Sample \(0 \le \alpha < 1\) and fail immediately if \(\alpha > \tp(n_i, n_{i-2})\). Fail if the flag \(a\) does not lie along a bigon. The size \(d\) of the face which would be produced by bigon deletion is the sum \(|f(\tau\sigma(a))|+|f(\sigma^3\tau(a))| - 2 = d\). Sample \(0 \le \beta < 1\) uniformly and fail if \(\beta > (d-1)^{-1}\). Otherwise, return \(D_{i+1} := \RIIm(D_i,a)\). The probability that this transition is chosen and succeeds if \(f(a)\) is a bigon and \(d > 1\) is,
  \begin{equation*} \tp(n, n-2) \frac{p_2}{2} = \min\left\{1,\frac{g_n}{g_{n-2}}\right\}\frac{p_2}{2(d-1)} . \end{equation*}
  
\item[{\([3]\)}] Triangle flipping, with probability \(p_3\). Fail if the flag \(a\) does not lie along a nondegenerate triangle. Otherwise, return \(D_{i+1} := \RIII(D_i,a)\). The probability that this transition is chosen and succeeds if \(a\) lies along a nondegenerate triangle is \(p_3\).

\end{itemize}

\subsection{Wang Landau tuning}
\label{sec:wanglandau tuning}
Before sampling, we have to gather data for \(g_n\) (the approximate enumeration) via a tuning algorithm with parameters \(\ell\), the smallest size diagram to allow in the sample space (always in this article \(\ell = 1\), as otherwise it is not necessarily clear if the Markov chain is ergodic), \(L\), the largest size diagram in the sample space, \(\epsilon\), which describes the desired flatness of the sampling histogram, and \(\Delta\), a threshold for flatness of a histogram of occurrences.

A starting point \(D_0\) in the sample space of diagrams is chosen; our algorithm starts with the figure-eight diagram in Figure~\ref{fig:figeight}. The values \(g_n\) are initialized to \(0\). Finally, a scaling factor \(f\) is initialized; we start it at \(f = 1\).

The algorithm then proceeds as follows.
\begin{enumerate}
\item If \(f < \epsilon\), terminate.
\item A histogram \(H = (H_n)_{n=\ell}^{L}\) of bins \(\ell\) to \(L\)
  inclusive is initialized empty. This histogram will track the occurrences of
  diagrams of size \(n\) at each step of the Markov chain.
\item Step, via the Wang-Landau weighted algorithm described above, producing
  \(D_{i+1}\) from the current \(D_i\). 
\item Every $S_0$ steps, increment \(H_{|D_{i+1}|}\) by 1, and increment \(g_{|D_{i+1}|}\) by \(f\).
\item Every $S_1$ steps, check if the histogram is \(\Delta\)-flat, \textit{i.e.}, check if
  \begin{equation} \frac{\min{H}}{1-\Delta} > \frac{\sum_{i=\ell}^{L}{H_{i}}}{L - \ell} >
    \frac{\max{H}}{1+\Delta}. 
   \end{equation} If so, let \(f := f/2\) and proceed with step (1).
  Otherwise repeat step (3).
\end{enumerate}
The numbers \(S_0,S_1\) are somewhat arbitrary, however \(S_1\) should be chosen to allow sufficient time for the chain to diffuse over all sizes. We note that, especially for large spreads of \(\ell\) and \(L\), the tuning phase may be time-consuming.  As the algorithm is tuning and \(g\) is being updated, the chain is \emph{not} ergodic as it does not satisfy the detailed balance condition. However, the tuning phase approaches a limiting state where \(g_{m+1}/g_m \approx k_{m+1}/k_m\). At that point, one may cease updating \(g\) and use those now fixed probabilities. This chain will satisfy detailed balance and so be ergodic. This tuned \(g\) data may then be reused for a number of further simulations without recalculation.

\begin{reptheorem}{thm:wlergodic}
   Let $N \in \N$ and $D$ be a plane curve with \(1 \leq n \leq L\) vertices and let $g_n$ be fixed. The Markov chain described in Section~\ref{sec:wanglandau} has stationary distribution given by 
  \[ \Prb(D) \propto \frac{1}{g_{n}} \]
  Since $g_n \approx k_n$, plane curves are sampled uniformly within a given size, and approximately uniformly across sizes.
\end{reptheorem}

\begin{proof}
  We will begin by assuming that
  \begin{equation} \Prb(D) \propto \frac{1}{g_{|D|}}. \end{equation}
  where $g_n$ data is fixed from the tuning phase of the algorithm.

  Let \(D\) be a rooted plane curve of \(n\) vertices and \(a\) be the root flag in \(D\). Observe first that the three pairs of reversing transitions \((\RIp, \RIm)\), \((\RIIp, \RIIm)\), and \((\RIII, \RIII)\) all change the number of vertices by distinct complementary amounts; hence any two diagrams can be related by at most one pair of transitions, or, if their vertex counts agree, a re-rooting. The detailed balance proofs are nearly identical to the Boltzmann case, substituting in the new value of \(\Prb\) and the new transition probabilities. Additional details can be found in Appendix~\ref{sec:detailedbalanceproofs}.
Provided \(p_1, p_2, p_3 > 0\), \((p_1+p_2+p_3) < 1\), and \(\ell = 1\), the Markov chain can reach all plane curves (as all plane curves can be changed through a crossing-non-increasing pathway to the curve with one crossing). Hence in this case, the Markov chain is ergodic.
\end{proof}

By Corollary~\ref{cor:noreroot}, we are able to simplify the Markov chain by re-rooting between steps, and only using the five Reidemeister transition operations as we did in the Boltzmann case. Again, in our simulations (detailed in the next section) we chose \( (p_1,p_2,p_3)=(2/5,2/5,1/5) \).

\section{Simulations and Data}
\label{sec:data}

We implemented both MCMC samplers in \texttt{c++}. Plane curves are stored as combinatorial maps; a collection of vertices, edges, and flags with bidirectional references between flags and their vertices, as well as flags and their edges. At each step, an flag is selected from the diagram at random; this takes \(O(1)\) time, and the parameter \(0 \le \alpha < 1\) is sampled uniformly at random. Both of the \(\RI\) moves, as well as the \(\RIII\) move take constant time. The \(\RIIp\) move requires an extra random number \(0 \le \gamma < 1\) which determines the second flag for the transition and is performed in constant time. The \(\RIIm\) move requires both the sampling of the additional number \(0 \le \beta < 1\) as well as a count of the face sizes diagonal to the bigon. A plane curve has average face degree strictly increasing and limiting on \(4\), so counting a face size requires a constant number of operations on average.

For low numbers \(n\) of crossings, it is also possible to sample plane curves uniformly through rejection sampling. A single sample is produced by sampling 4-valent maps uniformly until a plane curve is obtained. The maps are sampled via Gilles Schaeffer's bijection with blossom trees~\parencite{Schaeffer1997} using his \texttt{PlanarMap} software\parencite{SchaefferPlanarMap,PlanarMapLib}. The rejection step is simple, but plane curves are exponentially rare~\parencite{Schaeffer2004}, making this approach ineffective even for relatively small sizes: On a quad core 3.4Ghz Intel i5-7500 machine, sampling \(10^5\) 10-crossing curves takes 4.9 seconds, but sampling the same number of 100-crossing curves takes 712.6 seconds. For comparison, it only takes 21.9 seconds to sample \(10^5\) 100-vertex 4-valent maps.

In this section, we examine data from our simulations. First, we examine the distributions of plane curve sizes that our implementations produce. Then, we check how well our Wang-Landau implementation converges to the uniform distribution across fixed sizes by comparing statistics to the rejection sampler. The data of these sections are based on the following sampler runs:
\begin{enumerate}
\item \emph{Wang-Landau (WL)}. After tuning a Wang-Landau Markov chain to \(f < 10^{-8}\) with histogram flatness threshold \(\Delta = 0.99\), we drew a total of \(2 \times 10^7\) samples for sizes \(1 \le n \le 500\) with \(10^3\) steps between each sample. The tuning phase took approximately 500 minutes, and the sampling took 160 minutes. The sampled size probability distribution for this run is presented in Figure~\ref{fig:comparewlprb}. The minimum number of samples for any size is \(3.8933 \times 10^4\).
\item \emph{Rejection}. Using a rejection sampler for plane curves, we gathered \(10^4\) samples of plane curves with \(n\) vertices, for each \(n = 5m\) from \(5\) to \(100\). This took approximately \(40\) minutes.
\item \emph{All 4-valent maps}. Using a uniform sampler for all 4-valent maps, we gathered precisely \(10^4\) samples of 4-valent maps with \(n\) vertices, for each \(n = 5m\) from \(5\) to \(100\). This took approximately \(30\) seconds.
\end{enumerate}
Finally, we examine how our Wang-Landau sampler augmented to sample knot diagrams compares to the rejection sampler, as another check on the theoretical limiting distribution and our implementation.

\subsection{Size distributions}
\label{sec:sizedist}

We first examine the sample histograms of the Boltzmann and Wang-Landau samplers. This serves several purposes: First, as we are unable to sample diagrams of given fixed size, we would like to know how frequently we will sample a particular size using these methods. Second, in order to avoid correlated samples we would like our Markov chain to explore the full range of lengths frequently. Last, we can use comparisons of the sampled size distribution to better understand the counting sequence of plane curves.

\subsubsection{Boltzmann sampler implementation and the Boltzmann parameter \(z\)}
\label{sec:bzparams}

The MCMC sampler approximating the Boltzmann distribution on plane curves (described in Section~\ref{sec:boltzmann}) samples from a distribution,
\begin{equation} \Prb(D) \propto z^{|D|}, \end{equation}
where the parameter \(z\) affects the size of plane curves produced. Hence the probability of sampling any plane curve of size \(n\) is proportional to \(k_nz^n\). As noted above, we choose \(z < 1/\mu\) and so by Equation~\ref{eqn:asympkn} this probability is asymptotic to \((\mu z)^nn^{\gamma-2}\). Note that \(\gamma\) is expected to be \emph{negative} so any choice of \(z < \mu^{-1} \approx 0.0876\) will not, \textit{a priori} produce a finite local maximum.


To observe this we ran a number of experiments sampling \(10^6\) plane curves of a maximum size of \(200\) crossings with varying \(z\). These data produce the approximate size distributions of Figure~\ref{fig:bzsizehist}. This figure implies that it is difficult to pick \(z\) to obtain samples at large size while still sampling many objects of small size.

\begin{figure}[hbtp!]
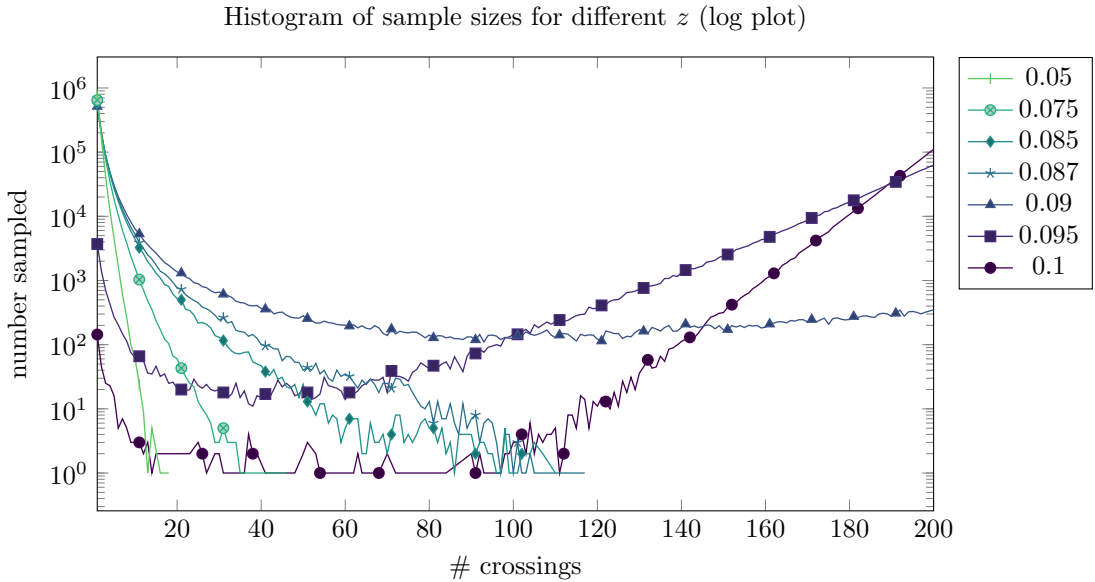

  \centering
  \includestandalone{mc_hist_1_100_zvar}
  \caption{Histogram of samples from various MCMC runs, with different \(z\)-values. The maximum-size plane curve was 200 vertices. The mixing time was \(10^3\) steps, and a total of \(10^6\) samples were drawn. The run for \(z=0.05\) took 65 seconds to complete, the run for \(z=0.1\) took 93 seconds, and run time generally increased with \(z\).}
  \label{fig:bzsizehist}
\end{figure}

We note that we want our algorithm to return to small sizes on a regular basis. First, returning to small sizes ``erases'' the entire object before producing a new sample. Second, the best known result on ergodicity relies on paths through small plane curves. Not enough is known about the connectivity of the space of plane curves to alleviate these concerns, providing strong justification for using the Wang-Landau sampler variant instead.



\subsubsection{Wang-Landau implementation}
\label{sec:wanglandaudata}

In comparison with the MCMC Boltzmann sampler, the tuning phase of the Wang-Landau sampler ensures that the size distribution sampled is actually approximately flat. Figure~\ref{fig:limiting_lgg} demonstrates how values of \( g_n \) approach the exact numbers of rooted plane curves, \( k_n \). Figure~\ref{fig:comparewlprb} presents the size distribution of a run of \(2 \times 10^7\) samples of size between \(1\) and \(500\) produced after tuning to \(f < 10^{-8}\) with \(\Delta = 0.99\).

\begin{figure}
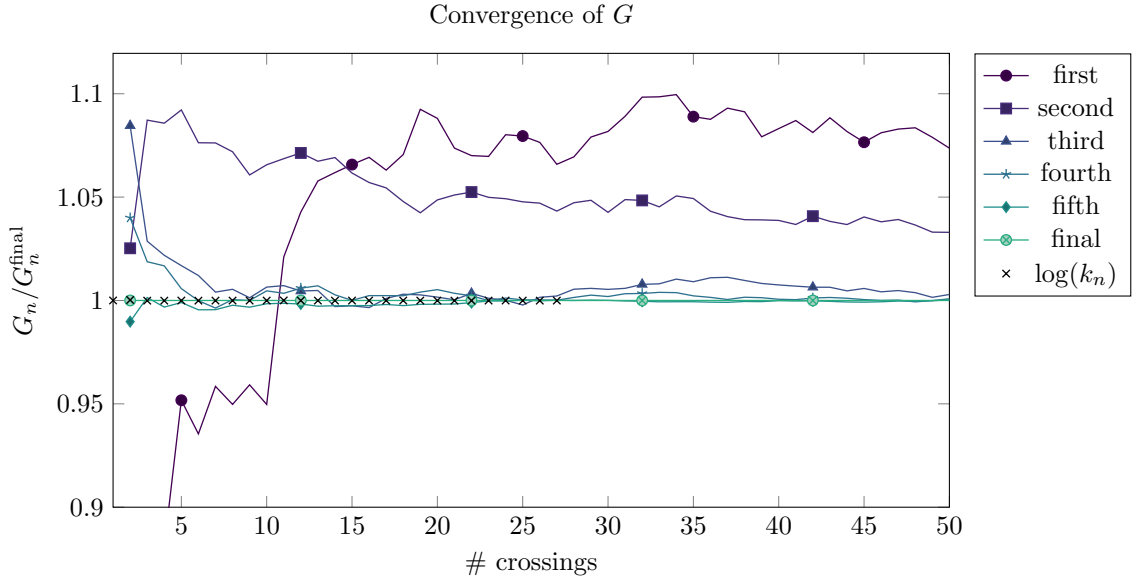

  \centering
  \includestandalone{mc_wl_flatten}
  \caption{As the tuning phase proceeds, the values \(G_n = \log g_n\) converge to \( G_n^{final} \). We also observe that \(G_n^{final} \approx \log k_n\) for available exact enumeration data, \(n \leq 27\).}
  \label{fig:limiting_lgg}
\end{figure}

\begin{figure}
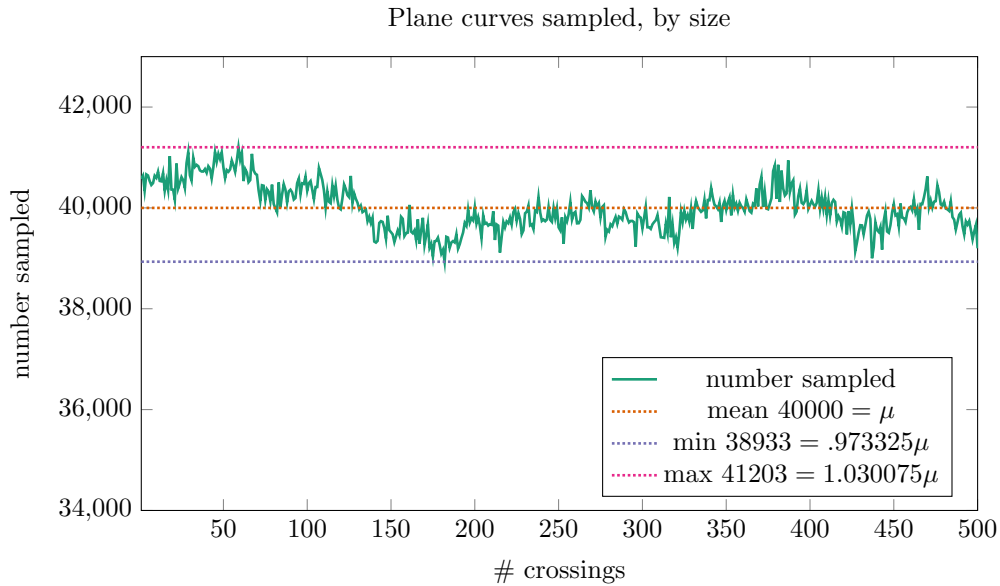

  \centering
  \includestandalone{wl_hist_counts}
  \caption{Distribution of plane curve sizes from a run of the Wang-Landau sampler. The sample histogram flatness is \(\max\{0.973325, 1- (1.030075-1)\} = 96.9925\%\).}
  \label{fig:comparewlprb}
\end{figure}

As mentioned, the Wang-Landau algorithm tuning step provides an estimation of ratios in the counting sequence for plane curves. We can use this data to provide approximations to the numbers of curves of given sizes. Using the tuning data for the run above (\(\ell=1, L=500, f < 10^{-8}, \Delta=0.99\)) we obtained approximate counts for \(n = 1\) to \(n = 27\) and can compare the approximations to the precise counts from~\parencite{ZinnJustin2009} in Table~\ref{tab:wltunecounts}. All approximate counts are within \(0.004\) of their exact value.

\begin{table}[hbtp!]
  \centering
\vspace{1em}
\includestandalone{wl_kn_est_table}
  \caption{Comparison of counts of rooted plane curves from the Wang-Landau
    tuning step with \(\ell=1, L=500, \varepsilon=10^{-8}, \Delta=0.99\), versus exact values gathered using another method~\parencite{ZinnJustin2009}.}
  \label{tab:wltunecounts}
\end{table}

We are able to obtain estimates for the unknowns \(\mu\) and \(\gamma\) in the predicted asymptotic growth formula \(k_n \sim C\mu^nn^{\gamma-2} \) from Wang-Landau \( g_n\) data. We attempted to use simple ratio estimates \(r_n = g_{n+1}/g_n \sim \mu \left(1 + \frac{\gamma-2}{n} \right)\), however the results are extremely noisy. Instead we used linear regression to fit to the model
\begin{align}
  \log g_n &= \log C + n \log \mu + (\gamma-2)\log n
\end{align}
Since we expect the \(g_n\) data to be noisier for larger \(n\) we fitted the above linear form to a subset of the data \( \{ g_n \mid 10 \leq n \leq n_{max}\} \). We then varied $n_{max}$ to get a rough estimate of corrections to the above asymptotic form. The resulting estimates (as functions of $n_{max}$) are shown in Figure~\ref{fig:lssq_est}. These results are consistent with earlier estimates \parencite{ZinnJustin2009} of $\mu \approx 11.416$ and $\gamma = -\frac{1+\sqrt{13}}{6} \approx -0.768$.

\begin{figure}
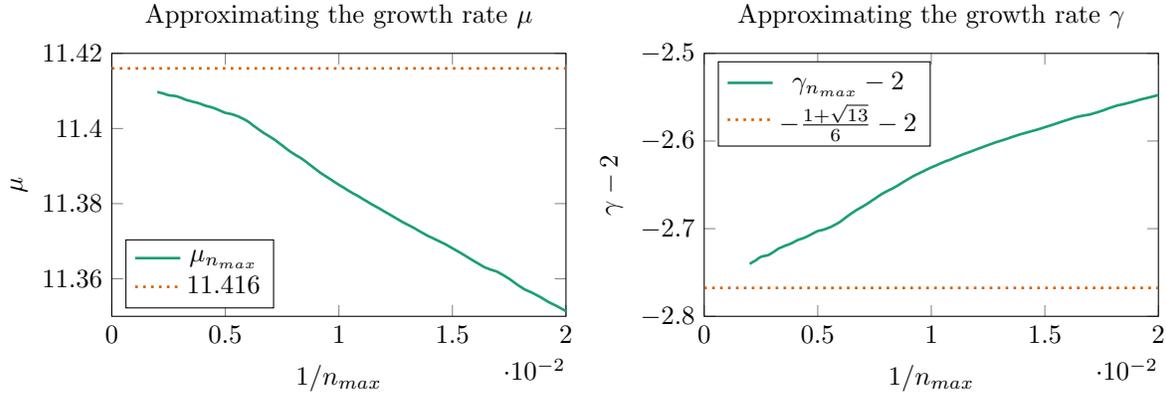

  \centering
  \includestandalone{mc_wl_lssq_mu}
  \includestandalone{mc_wl_lssq_gamma}
  \caption{Estimates of \(\mu\) and \(\gamma\) using least squares fits on \( \{g_n \mid 10 \leq n \leq n_{max} \} \) plotted as a function of \( n_{max}\).}
  \label{fig:lssq_est}
\end{figure}

\subsection{Face degrees in plane curves}
\label{sec:facedegrees}



In this section we seek statistic which distinguish plane curves from all 4-valent maps. A simple class of statistics to gather from random maps are vertex and face degrees. Plane curves, as a subclass of 4-valent planar maps, only ever have vertex degree 4, so only face statistics are nontrivial. Euler's formula implies that the average face degree for \emph{any} 4-valent map of \(n\) crossings is \(4n/(n+2)\), so this also cannot distinguish plane curves from its superclass. However, we will see that the \emph{distribution of face degrees} differs.

\subsubsection{Face degree probabilities}
\label{sec:facedegprob}

We check the counts of faces of fixed degree which appear (for a curve of \(n\) crossings, this takes \(O(n)\) time to compute as all plane curves have \(4n\) flags, each of which needs only be visited once). These quantities are expected to exhibit linear growth, in agreement with the results for a large number of map classes~\parencite{Liskovets99}. In fact, it is known:
\begin{theorem}
  Let \(k \ge 1\), and let \(p_{n,k}\) denote the probability that an arbitrary face of a random plane curve of \(n\) crossings has degree \(k\). Namely, notice that \((n+2)p_{n,k}\) is the expected number of degree \(k\) faces in a random plane curve of \(n\) crossings. Then 
  \begin{equation} 1 > \limsup_{n\to\infty}{p_{n,k}} \ge \liminf_{n\to\infty}{p_{n,k}} > 0. \end{equation}
\end{theorem}

\begin{proof}
  Certainly, \(0 \le p_{n,k} \le 1\) for all \(n, k\) as they denote probabilities; furthermore, \(\sum_{k=1}^{n} {p_{n,k}} = 1\). The pattern theorem for plane curves~\parencite{Chapman2016} says that, for any prime substructures \(T_k\) containing a \(k\)-gon, there are constants \(c_k > 0\), \(1 > d_k > 0\) and \(N \ge 0\) so that for all \(n \ge N\) the probability that an arbitrary plane curve of size \(n\) contains more than \(c_kn\) copies of \(T_k\) is at least \(1 - d_k^n\).

  So for \(n \ge N\), \((n+2)p_{n,k} > (1 - d^n)c_kn\). Solving for \(p_{n,k}\) and passing to \(\liminf\) yields,
  \begin{equation} \liminf_{n\to\infty} p_{n,k} > \liminf_{n\to\infty} (1 - d^n)c_k \frac{n}{n+2} = c_k. \end{equation}

  That \(1 > \limsup_{n\to\infty}{p_{n,k}}\) for any \(k\) follows from that for \(\ell \ne k\) the
  existence of a \(\ell\)-gon lowers the number of faces which may be \(k\)-gons. That all \(\liminf_{n\to\infty} p_{n,\ell} > 0\) yields the result.
\end{proof}

The following proposition summarizes the results we will use for planar 4-valent maps.

\begin{theorem}[Follows from~\parencite{Gao_1994,Liskovets99}]
  The limiting expected number (as maps grow large) of faces of degree \(k\) in a random 4-valent map is,
  \begin{equation} \frac{n+2}{k} [y^k]\left( (1/3)(1 + y/2)^{-\frac 12}(1 - 5y/6)^{-\frac 32} \right). \end{equation}
  \label{thm:limit4vslope}
\end{theorem}

\begin{proof}
  This is a rephrasing of selected results in~\parencite{Gao_1994,Liskovets99} in the language of 4-valent maps. Theorem 1 of \parencite{Gao_1994} says that the generating series of limiting probabilities \(q_k\) (as maps grow large) that the root vertex in an arbitrary map has degree \(k\) is,
  \begin{equation} \sum_{k\ge 0}{q_k y^k} = (1/12)(1 + y/2)^{-\frac 12}(1 - 5y/6)^{-\frac 32}. \end{equation}
  Duality of the class of rooted planar maps says the same result holds for faces. The bijection between \(m\)-edged rooted planar maps and \(m\)-faced rooted planar quadrangulations then says the same result holds for vertices in quadrangulations (see for instance the proof of Proposition 12 in~\parencite{Benjamini_2013}).

  Section 2.5 in~\parencite{Liskovets99} relates, for a quadrangulation of \(m\) vertices, the probability \(q_{k,m}\) that the root vertex has degree \(k\) to the probability \(p_{k,m}\) that an arbitrary vertex has degree \(k\) by,
  \begin{equation} p_{k,m} = \frac{4(m-2)}{m} \frac{q_{k,m}}k \sim 4\frac{q_{k,m}}{k}, \end{equation}
  whence \(p_{k} = \lim_{m\to\infty} p_{k,m}\) and \(q_{k} = \lim_{m\to\infty} q_{k,m}\) are related by \(p_{k} = 4q_k/k\). Noting that 4-valent maps with \(n\) vertices are dual to quadrangulations with \(n+2\) vertices yields the result.
\end{proof}

We have computed linear regressions using least squares for 4-valent maps sampled using the Schaeffer bijection, plane curves sampled using rejection, and plane curve sampled using our Wang-Landau sampler for face degrees from 1 to 9 and presented these data in Table~\ref{tab:facedegfits} alongside precise densities obtained via Taylor series expansion on the result of Theorem~\ref{thm:limit4vslope}. We present these data for small faces sizes in Figure~\ref{fig:facedegcounts}. In all cases, the data of the Wang-Landau sampler resides within the error bars of the rejection curve sampler data, and away from the arbitrary 4-valent map data.

\begin{table}
  \centering
  \renewcommand*{\arraystretch}{1.5}
  \begin{tabular}{|c|c|c|c|c|}
    \hline
    \(k\) & 4-valent \(p_k\), theor.                     & 4-valent \(p_k\)              & Rejection \(p_k\)               & WL \(p_k\)                    \\
    \hline
    1     & \(\tfrac {1}{3} = 0.\overline3\)             & \(0.3328 \pm 2\cdot10^{-4}\)  & \(0.3496 \pm 2\cdot10^{-4}\) & \(0.35036 \pm 4\cdot10^{-5}\) \\
    2     & \(\tfrac {1}{6} = 0.1\overline6\)            & \(0.1662 \pm 2\cdot10^{-4}\)  & \(0.1407 \pm 1\cdot10^{-4}\) & \(0.14056 \pm 3\cdot10^{-5}\) \\
    3     & \(\tfrac {13}{108} = 0.12\overline{037}\)    & \(0.1200 \pm 2\cdot10^{-4}\)  & \(0.1222 \pm 1\cdot10^{-4}\) & \(0.12257 \pm 3\cdot10^{-5}\) \\
    4     & \(\tfrac {55}{648} \approx 0.08488\)         & \(0.0846 \pm 1\cdot10^{-4}\)  & \(0.0831 \pm 1\cdot10^{-4}\) & \(0.08298 \pm 2\cdot10^{-5}\) \\
    5     & \(\tfrac {83}{1296} \approx 0.06404\)        & \(0.0641 \pm 2\cdot10^{-4}\)  & \(0.0663 \pm 2\cdot10^{-4}\) & \(0.06624 \pm 2\cdot10^{-5}\) \\
    6     & \(\tfrac {377}{7776} \approx 0.048448\)      & \(0.0484 \pm 2\cdot10^{-4}\)  & \(0.0500 \pm 1\cdot10^{-4}\) & \(0.04977 \pm 2\cdot10^{-5}\) \\
    7     & \(\tfrac {1751}{46656} \approx 0.03753\)     & \(0.03769 \pm 8\cdot10^{-5}\) & \(0.0395 \pm 1\cdot10^{-4}\) & \(0.03941 \pm 1\cdot10^{-5}\) \\
    8     & \(\tfrac {101}{3456} \approx 0.02922\)       & \(0.0292 \pm 1\cdot10^{-4}\)  & \(0.0305 \pm 1\cdot10^{-4}\)  & \(0.03059 \pm 1\cdot10^{-5}\) \\
    9     & \(\tfrac {115825}{5038848} \approx 0.02299\) & \(0.0232 \pm 1\cdot10^{-4}\)  & \(0.02419 \pm 8\cdot10^{-5}\) & \(0.02419 \pm 1\cdot10^{-5}\) \\
    \hline
  \end{tabular}\\
  \caption{Limiting face degree densities. Theoretical fits for 4-valent \(p_k\) come from Taylor series expansion of the result
    in Theorem~\ref{thm:limit4vslope}. Experimental columns are the slopes of linear functions fit to data using least squares.}
  \label{tab:facedegfits}
\end{table}

We note that: Random curve diagrams have fewer bigons and quadrangles than generic 4-valent maps. This phenomenon is specific to degrees 2 and 4 (at least for face degrees at most 9); every other degree face is more common in random plane curves.

\begin{figure}[hbtp!]
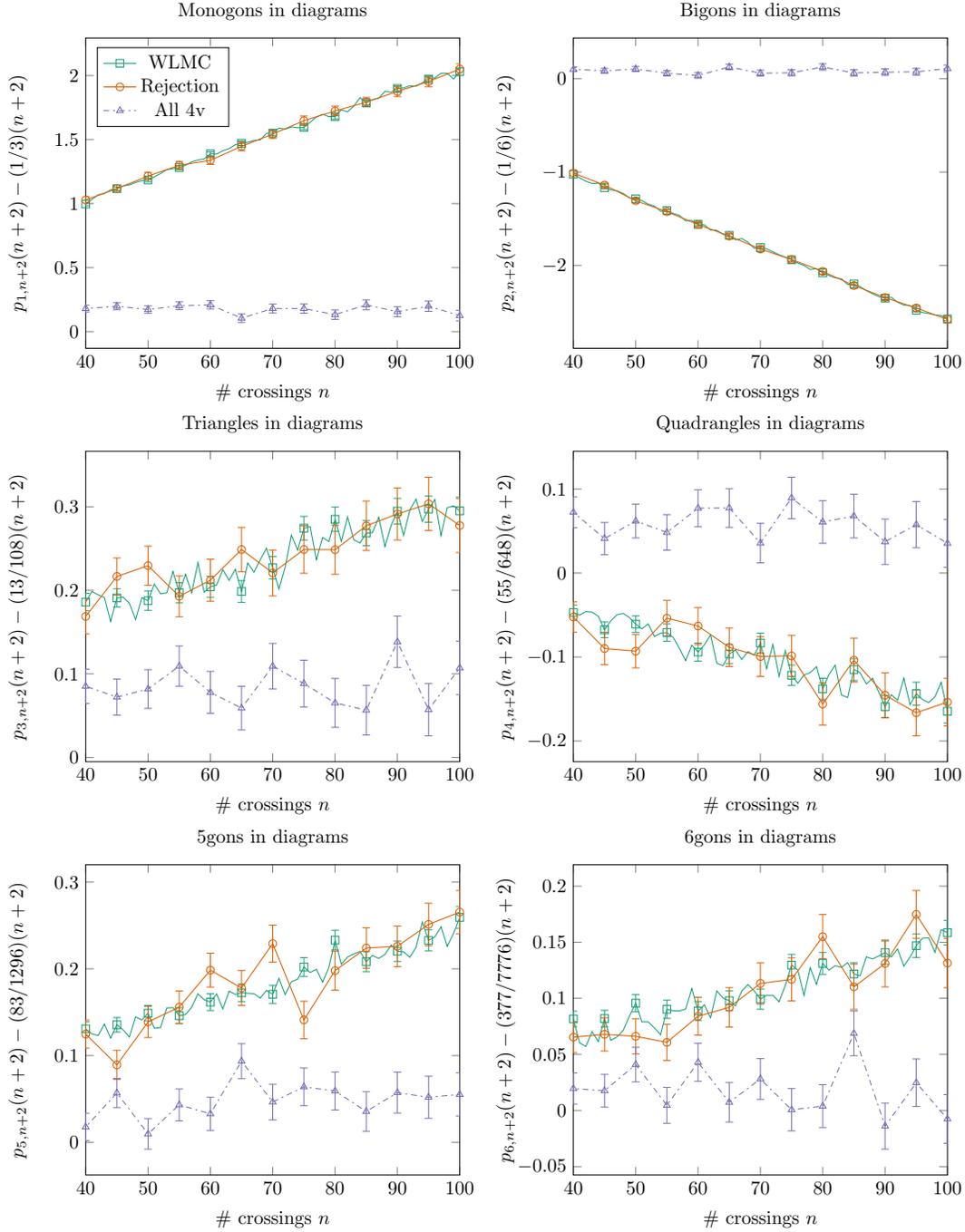

  \centering
  \includestandalone{face_d1_plot_deltas}
  \includestandalone{face_d2_plot_deltas}
  \includestandalone{face_d3_plot_deltas}
  \includestandalone{face_d4_plot_deltas}
  \includestandalone{face_d5_plot_deltas}
  \includestandalone{face_d6_plot_deltas}
  \caption{Average counts \(p_{k,n+2}(n+2)\) of faces of degrees 1--6.}
  \label{fig:facedegcounts}
\end{figure}

It is furthermore expected of classes of random maps that, for fixed size, number of faces of fixed degree \(k\) is a normally distributed statistic~\parencite{Drmota2013}. In Figure~\ref{fig:faced2dist} below, we compare distributions of different \(k\)-gon ratios for \(n=40\) and \(n=100\) crossings. As expected, the curves show a close similarity between the uniformly sampled curves and the Wang-Landau MCMC sampled curves. In the cases of \(1\)- and \(2\)-gons, it is easy to see a difference from the all 4-valent map sampler. In the case of larger faces, the differences in averages are on a much smaller order, and the curves are no longer possible to distinguish. In all cases, it seems that as the number of crossings \(n\) grows large, the distributions are approximately normal.

\begin{figure}[hbtp!]
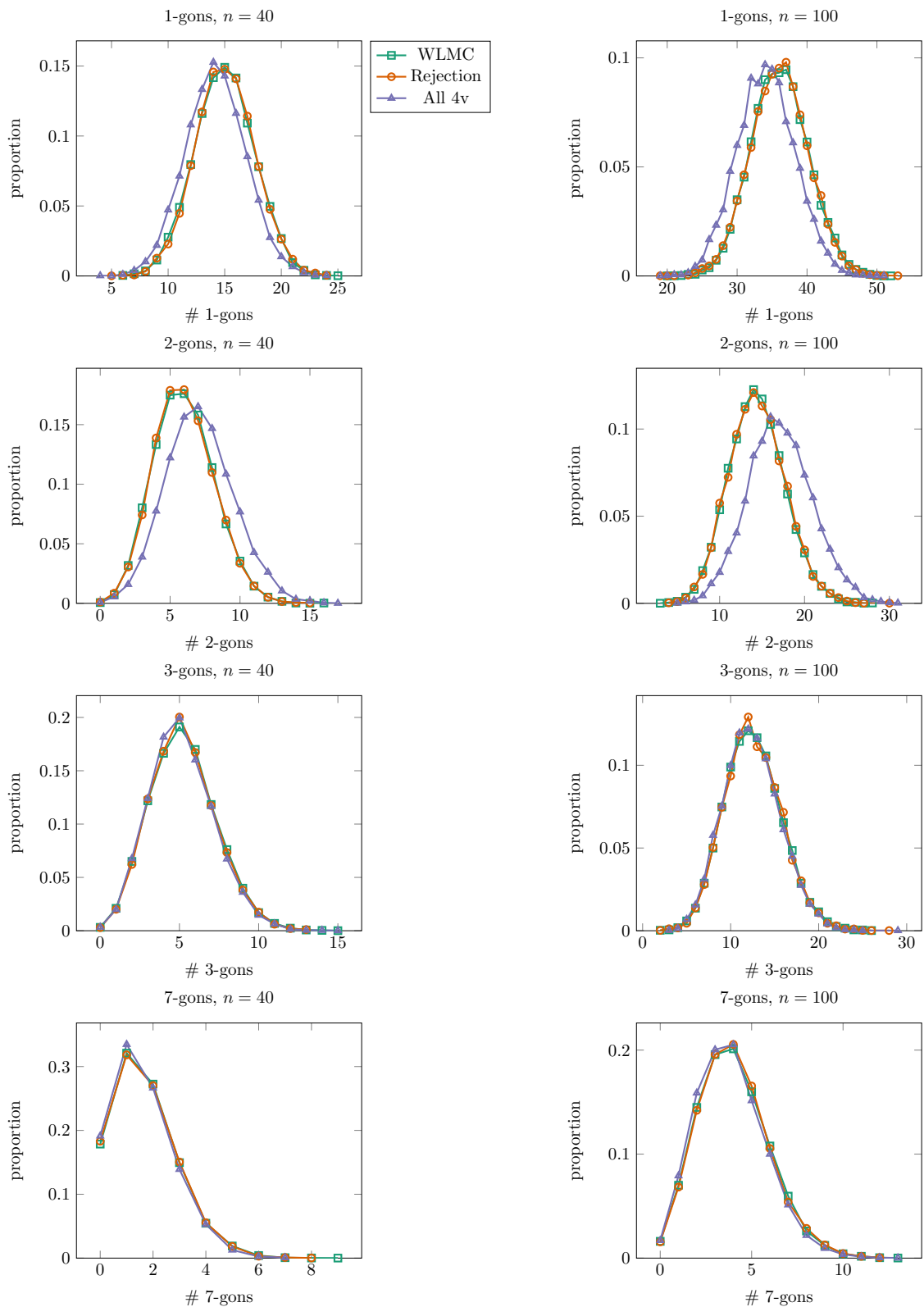

  \centering
  \begin{subfigure}[t]{0.45\linewidth}
    \includestandalone{hist_1gons_40}
  \end{subfigure}
  \hfill
  \begin{subfigure}[t]{0.45\linewidth}
    \includestandalone{hist_1gons_100}
  \end{subfigure}
    \begin{subfigure}[t]{0.45\linewidth}
    \includestandalone{hist_2gons_40}
  \end{subfigure}
  \hfill
  \begin{subfigure}[t]{0.45\linewidth}
    \includestandalone{hist_2gons_100}
  \end{subfigure}
    \begin{subfigure}[t]{0.45\linewidth}
    \includestandalone{hist_3gons_40}
  \end{subfigure}
  \hfill
  \begin{subfigure}[t]{0.45\linewidth}
    \includestandalone{hist_3gons_100}
  \end{subfigure}
    \begin{subfigure}[t]{0.45\linewidth}
    \includestandalone{hist_7gons_40}
  \end{subfigure}
  \hfill
  \begin{subfigure}[t]{0.45\linewidth}
    \includestandalone{hist_7gons_100}
  \end{subfigure}
  \caption{Distribution of \(k\)-gon counts for \(n=40\) and \(n=100\)
    crossings, for various \(k\).}
  \label{fig:faced2dist}
\end{figure}

We note here that the face degree probabilities are closely related to the probability that a given Markov chain transition succeeds on a given diagram, although it is actually the quantity \(q_{k,n}= \frac{k p_{k,n}(n+2)}{4n} \approx \frac{kp_{n,k}}4\) discussed prior which is at play (transitions by definition occur at the root along the root face). Namely, the probability that an \(\RIII\) move can succeed is \(q_{3,n+2} \approx 0.09193\), the probability that an \(\RIm\) operation will succeed is \(q_{1,n+2} = 0.08759\), and the probability that an \(\RIIm\) operation will succeed is \(q_{2,n+2} \approx 0.07028\) (\emph{not taking into account the extra Metropolis-Hastings step required to create large faces or sites which would create monogons}, see the description of move \(\RIIm\) in Section~\ref{sec:markovchain}). 

We reinterpret briefly the probabilities \(p_{k,n+2}\) in the context of random knot diagrams. We note that in the case of prime alternating diagrams, face degrees are related to the hyperbolic volume of the resulting knot~\parencite{Obeidin16}. We further note that a random curve has \(\approx 0.35(n+2) > (n+2)/3\) monogons says that a random knot diagram has, on average, at least \(.35(n+2)\) vertices that have \emph{no} impact on the knot type and could be immediately reduced by a Reidemeister I move. One half of all crossing assignments for bigon vertices can be reduced by Reidemeister II moves, so a random knot diagram will have around \(0.14(n+2)/2\) bigons that can be removed by Reidemeister II moves.


\subsubsection{Maximum face degree}
\label{sec:maxfacedeg}

Another quantity of interest in the study of planar maps is the maximum vertex degree \(\Delta_n\) and the maximum face degree \(\Delta_n^*\). As noted above, \(\Delta_n = 4\) because all vertices are 4-valent, so we examine expectation of \(\Delta_n^*\). It is expected that this quantity exhibits \(\Theta(\log n + \log \log n)\) growth as it does for general maps. A result of~\parencite{Gao2000} has that for general maps the expected maximum face degree is precisely:
\begin{equation}
  \Exp(\Delta_n^*) = \frac{\ln(n)- \frac 12\ln(\ln(n))}{\ln(6/5)}.
\end{equation}

We present the difference between expectations and that of general maps in Figure~\ref{fig:largestface}. From this, it does not appear that plane 4-valent maps exhibit the same behavior as the general map case: Indeed, as the bijection between arbitrary maps and 4-valent maps makes both faces and vertices into 4-valent map faces, the expected maximum vertex degree \(\Delta_n^*\) is in fact an expectation of a maximum \(\mathbb E (\max{(\Delta_n, \Delta^*_n)})\) over both vertex and face degrees of arbitrary maps. We have plotted histograms of the maximum face degree distribution in curves and 4-valent maps for fixed size in Figure~\ref{fig:maxfvdist}. These distributions are clearly not Gaussian. We note that the histograms of all cases look similar even though they have differing means (see Figure~\ref{fig:largestface}), this can be explained by noting that the difference between the means is small and growing very slowly in \(n\).

\begin{figure}[hbtp!]
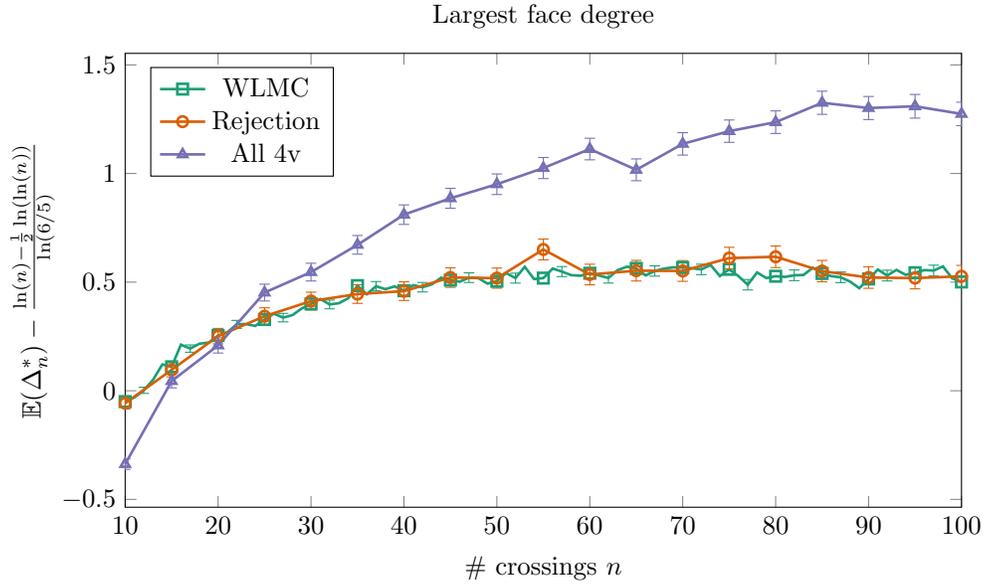

  \centering
  \includestandalone{largest_face_plot}
  \caption{Average size of largest face \(\Delta_n^*\).}
  \label{fig:largestface}
\end{figure}

\begin{figure}[hbtp!]
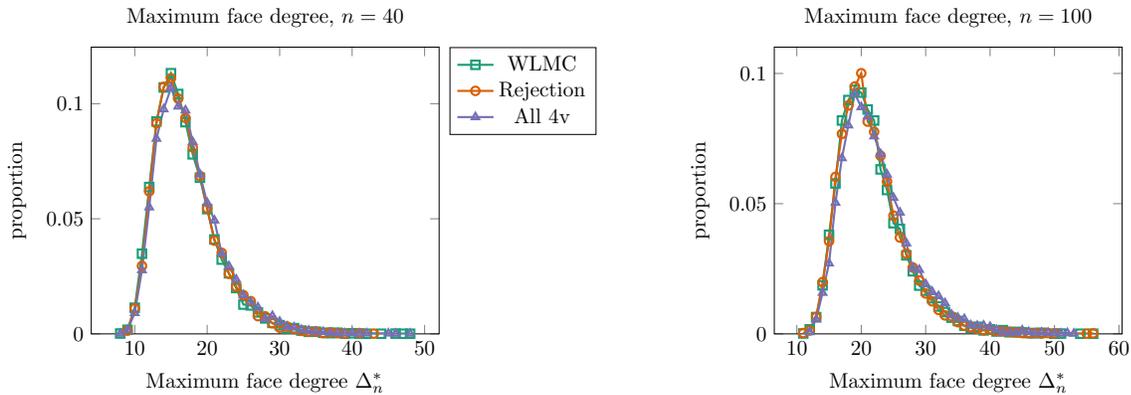

  \centering
  \begin{subfigure}[t]{0.45\linewidth}
    \includestandalone{maxfv_hist_40}
  \end{subfigure}
  \hfill
  \begin{subfigure}[t]{0.45\linewidth}
    \includestandalone{maxfv_hist_100}
  \end{subfigure}
  \caption{Distribution of largest face degree \(\Delta_n^*\) for \(n=40\) and \(n=100\)
    crossings.}
  \label{fig:maxfvdist}
\end{figure}

We present the data for \(\Exp(\Delta_n^*)\) of the Wang-Landau MCMC sampler up to 500 crossings---it is impractical to gather samples of this size from the rejection sampler---in Figure~\ref{fig:wllongtermmaxfv}. The \(\Theta(\log n)\) trend continues, as hypothesized. We note that the quantity \(\Exp(\Delta_n^*)\) is related to the success rate of the \(\RIIm\) transition, and suggests that the rejection rate for creating large faces through this transition is roughly bounded by \(1/\Theta(\log n)\).

\begin{figure}
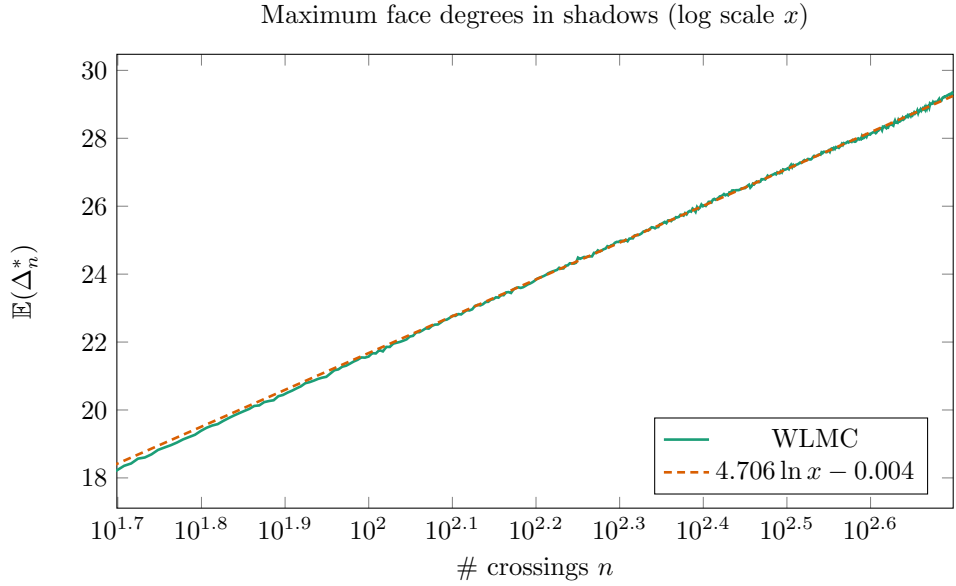

  \centering
  \includestandalone{wl_maxfv_longterm}
  \caption{Expected maximum face valence in plane curves appears to grow logarithmically.}
  \label{fig:wllongtermmaxfv}
\end{figure}

\subsection{Average Casson invariant}
\label{sec:v2ivt}

Our results above suggest that objects we sample using our Wang-Landau algorithm are giving the same statistics as those generated using rejection sampling. To further test this idea, we can compare further plane curve statistics.

Plane curves are equivalent to Arnol'd's spherical curves~\parencite{Arnold1995}. Thus, we can check \(\Z\)-valued spherical curve invariants for plane curves (these are not defined in the case of a general 4-valent map). These are both simple to compute and interesting knot theoretically: By following the crossings in order around the plane curve and counting those which are ``interlaced'', we are able to compute \(-\frac 12 (2St+J^+)\) in \(O(n^2)\) time (we direct the reader to~\parencite{Arnold1995} for more details). This curve invariant is in fact related to a knot invariant, the finite type invariant \(v_2\), also called the Casson invariant: \(-\frac 18 (2St + J^+)\) is the expected value of \(v_2\) over \emph{all possible over-under crossing sign assignments} to the plane curve. A comparison of \(\Exp(v_2)\) data for Wang-Landau MCMC and rejection samples are presented in Figure~\ref{fig:avgstj}; the data are remarkably consistent.

\begin{figure}[hbtp!]
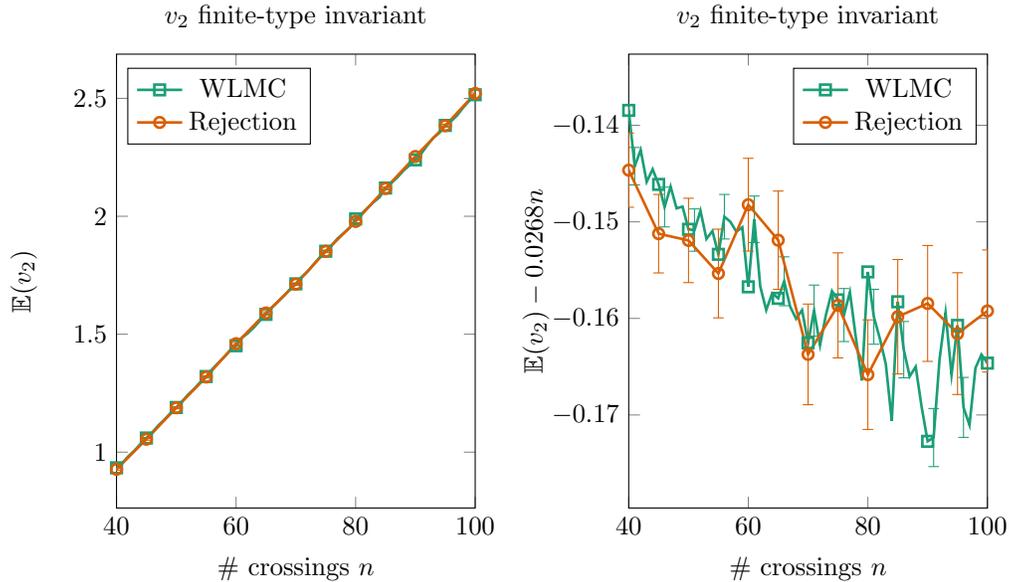

  \centering
  \includestandalone{avg_v2_plot_slopy}
  \includestandalone{avg_v2_plot}
  \caption{(Left) Average \(-\frac 18 (2St + J^+) = \Exp(v_2)\). (Right) Same statistic with approximate leading linear behavior subtracted to amplify detail. As this statistic is not
    well-defined for an arbitrary \(4\)-valent map, that data is not present.}
  \label{fig:avgstj}
\end{figure}

We present the data for \(\Exp(v_2)\) up to \(500\) vertices generated from the Wang-Landau algorithm in Figure~\ref{fig:wllongtermv2inv}. Our data suggest that the average \(v_2\) invariant grows linearly with the number of crossings with a limiting slope of \(0.0268 \pm 0.001\). As a comparison, Even-Zohar et al.~\parencite{EvenZohar2014} prove that for the Petaluma model of random knots, the expectation of \(v_2\) grows quadratically with size with leading coefficient \(1/24 \approx 0.0417\). As the petal diagrams of size \(n\) of the Petaluma model can be viewed as ``star diagrams'' with \(\frac{n^2-3n}{2}\) crossings, size in the Petaluma model is related quadratically to that of our model. Thus, it is reasonable that our average \(v_2\) data should grow linearly. It is also expected that the distribution of \(\Exp(v_2)\) over curves of a fixed size tends to be normal; we present histograms of this in Figure~\ref{fig:v2ivtdist} that appears to agree with this hypothesis.

\begin{figure}[hbtp!]
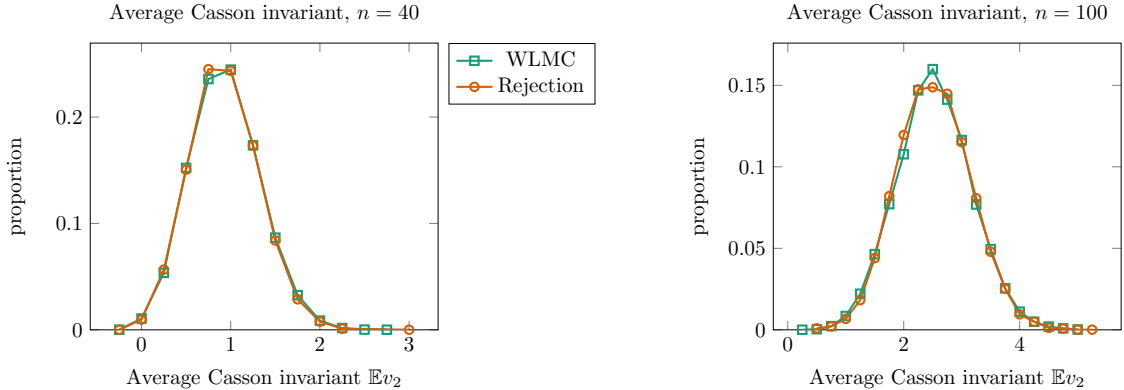

  \centering
  \begin{subfigure}[t]{0.45\linewidth}
    \includestandalone{v2inv_hist_40}
  \end{subfigure}
  \hfill
  \begin{subfigure}[t]{0.45\linewidth}
    \includestandalone{v2inv_hist_100}
  \end{subfigure}
  \caption{Distribution of mean \(v_2\) invariant for \(n=40\) and \(n=100\)
    crossings.}
  \label{fig:v2ivtdist}
\end{figure}


\begin{figure}
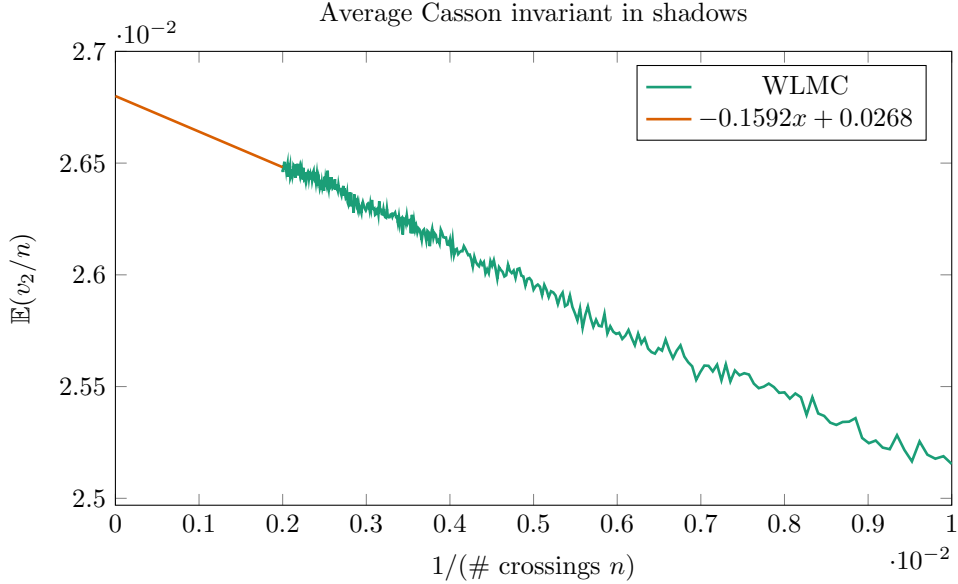

  \centering
  \includestandalone{wl_v2inv_longterm}
  \caption{The average \(v_2\) invariant appears to have linear dominant order growth. Here we present its ratio by number of crossings for finer resolution, where a na\"ive least squares linear regression suggests a slope of \((2.68 \pm 0.1)\times 10^{-2}\).}
  \label{fig:wllongtermv2inv}
\end{figure}

\subsection{Knotting probabilities}
\label{sec:knotprobs}

The main aim of constructing these Markov chains was to provide an efficient mechanism to sample large random knot diagrams. We have used the Wang Landau chain to sample random plane curves which we then use to construct random knot diagrams. This last step is done by mapping each vertex to an over or under crossing uniformly at random. We then use knot invariants to determine the topology of the resulting diagram and we compare the resulting knotting probabilities to those found using the rejection sampler.

In~\parencite{Chapman2016} the first author gathered data for knotting probabilities in knot diagrams of size up to 100 crossings. We have used our Wang-Landau sampler to produce random knot diagrams and classify them by HOMFLY-PT polynomial using \texttt{lmpoly}~\parencite{Ewing_1991,Ewing1997} of size up to 230 crossings. There are two reasons that we have not investigated beyond this size. First, the version of \texttt{lmpoly} included with plCurve~\parencite{PlCurve} will not compute HOMFLY-PT polynomials of diagrams of more than \(255\) crossings. Second, the difficulty of computing the HOMFLY-PT polynomial increases dramatically with the number of crossings; the algorithm employed by \texttt{lmpoly} has exponential run time\parencite{Ewing_1991} and computation is known to be NP-hard. We gathered data using two runs (from the same tuning data for \(1 \le n \le 230\) with \(f < 10^{-8}\) and \(\Delta = 0.99\)). In each run we took \(10^3\) steps between samples, and a total of \(1.15\times 10^{7}\) diagrams were sampled. Each run took approximately 9 hours to compute, with HOMFLY-PT polynomial calculation being the primary bottleneck. The histograms for sizes sampled are presented in Figure~\ref{fig:knotsamphist}. We note some caveats for these experiments:
\begin{enumerate}
\item The HOMFLY-PT polynomial is \emph{not} sufficient to distinguish all knot types, and indeed there are infinite families of knots demonstrating this~\parencite{Kanenobu_1986}. However, it is still unknown whether the unknot \(0_1\) is the only knot with trivial HOMFLY-PT polynomial; this is related to the still open Jones Conjecture~\parencite{Kauffman87}.
\item To facilitate a greater number of samples, we imposed a timeout on \texttt{lmpoly}'s HOMFLY-PT calculation of 10ms. See Figure~\ref{fig:homflyfailure}. Knots whose HOMFLY-PT failed to be calculated before this cutoff are still counted and categorized as unclassified. Hence it is possible that probabilities presented in our data which follows are smaller than the actual. There is evidence that these failures are rare for simple knots such as the trefoil and the unknot --- see Figure~\ref{fig:homfly_knots_failure}. We note that the uniform data presented alongside our MCMC sampled data was gathered at a larger timeout.  
\item We ignore chirality in these data. A chiral knot is a knot which is different than its mirror, such as the trefoil \(3_1\), while an achiral knot like the figure-eight knot \(4_1\) is equivalent to its mirror image. By symmetry (\emph{i.e.}\ by flipping all crossing signs simultaneously), it can be seen that the probability of drawing a chiral knot is equivalent to that for its mirror image. For the case of the chiral granny \(3_1 \# 3_1\) and achiral square \(3_1 \# 3_1^*\) composite knots, our data suggests that the square knot is as likely as either chiral image of the granny knot; we hence suppress data for the granny knot.
\end{enumerate}

\begin{figure}
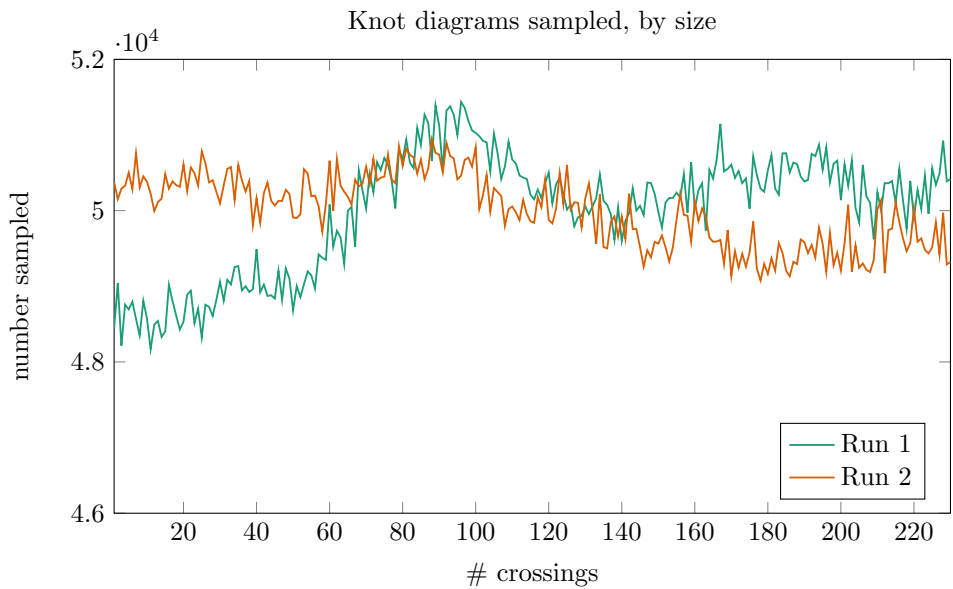

  \centering
  \includestandalone{knotsamphist}
  \caption{Histograms of sample sizes for two datasets used to examine random knotting.}
  \label{fig:knotsamphist}
\end{figure}

\begin{figure}
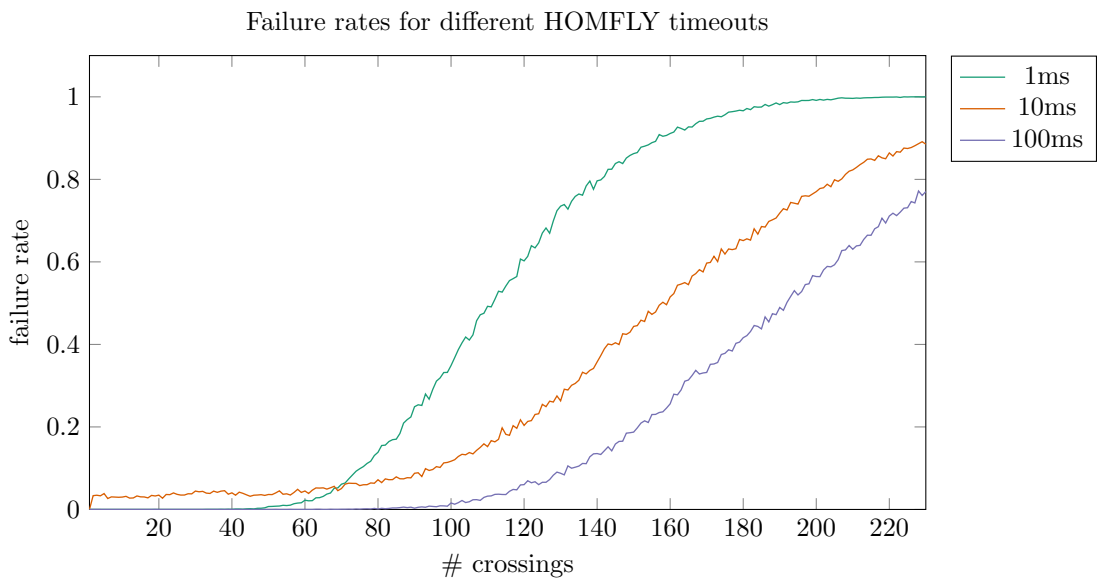

  \centering
  \includestandalone{homfly_timeout_null}
  \caption{The time taken to compute the HOMFLY-PT polynomial increases dramatically with the complexity of the knot diagram. Accordingly, we imposed a time-cutoff on these computations and if the invariant was not computed before the indicated time then it was left unclassified. In this figure we show the proportion of knot diagrams that failed to be classified as a function of the number of crossing and the cutoff time. }
  \label{fig:homflyfailure}
\end{figure}

\begin{figure}
  \centering
  \includestandalone{01_timeouts}
  \includestandalone{31_timeouts}
  \caption{The probability of a knot being classified as an unknot and a trefoil as a function of the number of crossings and the time-cutoff of the HOMFLY-PT computation. This data shows that increasing the cutoff does not significantly affect this probability. This is consistent with the hypothesis that most HOMFLY-PT computation failures are for complicated knots.}
  \label{fig:homfly_knots_failure}
\end{figure}

\begin{figure}
  \centering
  \includestandalone[width=2in]{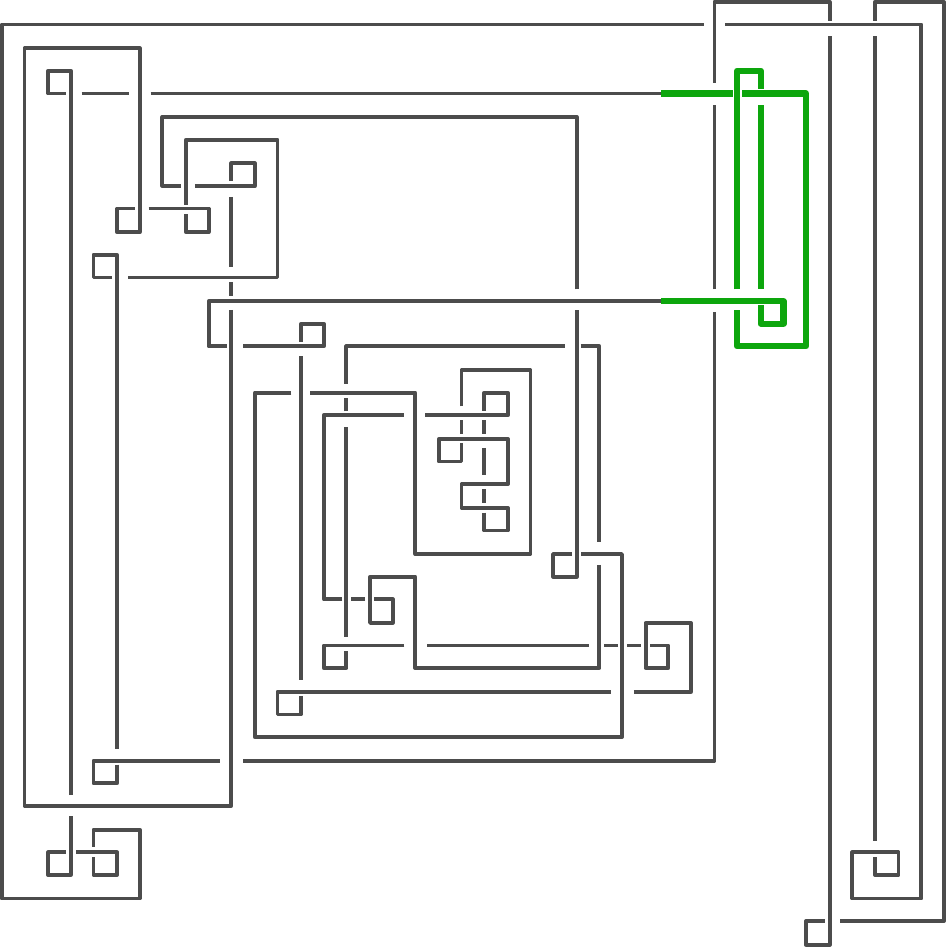}
  \caption{A random trefoil knot diagram of 50 crossings; the knotted portion of the curve is highlighted. We see that the knotted portion is quite small and this is expected to be typical.}
  \label{fig:localtrefoil}
\end{figure}

It is believed that for lattice models of random knots the number \(\kappa_n(K)\) of knot diagrams with fixed knot type \(K\) has asymptotic growth rate 
\begin{equation} 
\kappa_n(K) \sim C_K\tau_K^n n^{\alpha_K+N_K}, 
\end{equation}
where \(C_K\) depends on the knot type and \(N_K\) is the number of prime
components making up the knot type \(K\). It is believed that the constant
\(\tau_K\) does not depend on the knot
type~\parencite{Rensburg_1991,Deguchi_1997,Orlandini1998}. It has been proved
that \( \tau_0 \) exists for many lattice models (this follows from standard
supermultiplicativity arguments) and also for random knot
diagrams~\parencite{Chapman2016}, however it is still an open problem to prove
the existence of \(\tau_K\) for any other knot type. It is known, however, that
\( \tau_0 \) is strictly smaller than \( 2\mu \) for random knot diagrams~\parencite{Chapman2016}; which is comparable to a similar result for self-avoiding polygons~\parencite{Sumners_1988, Pippenger89}. There is strong numerical evidence for self-avoiding polygons that the exponent \(\alpha_K \) is independent of knot type~\parencite{Orlandini1998,Rensburg2011}. Consequently we conjecture that for random knot diagrams that 
\begin{equation} 
\kappa_n(K) \sim C_K \tau_0^n n^{\alpha+N_K}, 
\end{equation}
where \(\alpha = \gamma-2\), where \(\gamma\) is the same ``universal'' critical exponent~\parencite{Schaeffer2004} in the asymptotic formula for plane curves. This asymptotic form is consistent with the idea that the knotted portion of a random knot diagram is localized --- see Figure~\ref{fig:localtrefoil}.

Under this assumption, the probability that a random knot diagram of size \(n\) exhibits knot type \(K\) scales as
\begin{equation}
  p_n(K) = \kappa_n(K)/\kappa_n \sim D_K \rho^n n^{N_K},
  \label{eqn:knotprob}
\end{equation}
where \(0 < \rho < 1\). We plot knot probability data from both our Wang Landau sampler and rejection sampling in Figure~\ref{fig:knotcompares}. We see that both sampling methods agree and that the data is consistent with the scaling form in Equation~\ref{eqn:knotprob}. In particular, in Figure~\ref{fig:loglinear_knots}, we plot the logarithm of the knotting probabilities divided by $n^{N_K}$ and see that the resulting slopes are extremely similar. A simple linear regression of this data shows that $\rho \approx 0.95$. This is evidence that the growth rate of random knot diagrams of fixed knot type $K$ is independent of $K$ and that 
\begin{align}
\tau_K = \tau_0 \approx 2\times 11.41 \times 0.95 = 21.7.
\end{align}
The authors intend to test this hypothesis further in future work.

\begin{figure}
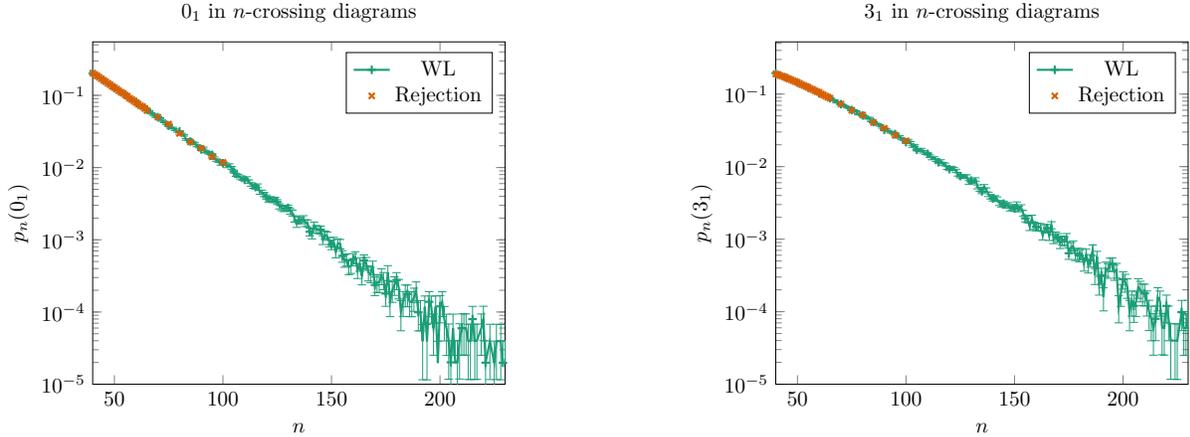

  \centering
  \begin{subfigure}[t]{0.45\linewidth}
    \includestandalone{01fit}
  \end{subfigure}\hfill
    \begin{subfigure}[t]{0.45\linewidth}
    \includestandalone{31fit}
  \end{subfigure}\hfill
  \caption{Plots for a WL sample against the data for a rejection sampler in~\parencite{Chapman2016}. Plots for other knot types are similar, and rejection data is consistent with WL sampled data throughout.}
  \label{fig:knotcompares}
\end{figure}

\begin{figure}
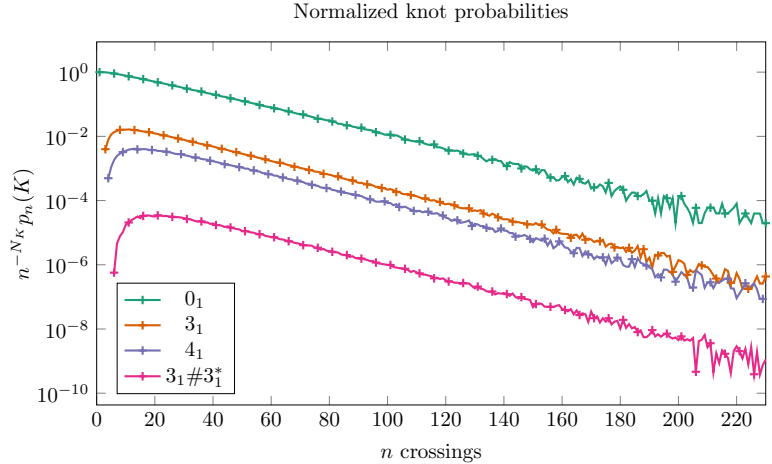

  \centering
  \includestandalone{knot_slopes}
  \caption{After normalizing by a factor of \(n^{-N_K}\) (where \(N_K\) is the number of prime components of the knot type \(K\), knot probabilities are all approximately log-linear.}
  \label{fig:loglinear_knots}
\end{figure}

As a final comparison to the uniform sampler data, we compute ratios of knot probabilities as in~\parencite{Rensburg2011}. Namely, the expected growth rates of knot probabilities has that, for two knot types \(K\) and \(L\), the ratio of probabilities should obey,
\begin{equation} 
\frac{p_n(K)}{p_n(L)} \sim \frac{D_K \rho^n n^{N_K}}{D_L \rho^n n^{N_L}} = \frac{D_K}{D_L}n^{N_K-N_L}
\end{equation}
Hence we expect $p_n(K)/p_n(L) \cdot n^{N_L - n_K}$ to tend to a constant as n increases. We plot this data for ratios of prime knots in Figure~\ref{fig:primeknotrat}, and for ratios of unknots to prime knots Figure~\ref{fig:unknotrats}. We also show the ratios of square knots to trefoils and unknots in Figure~\ref{fig:squarerats}. While this ratio data is noisy for larger $n$ it is consistent with the above scaling form.

\begin{figure}
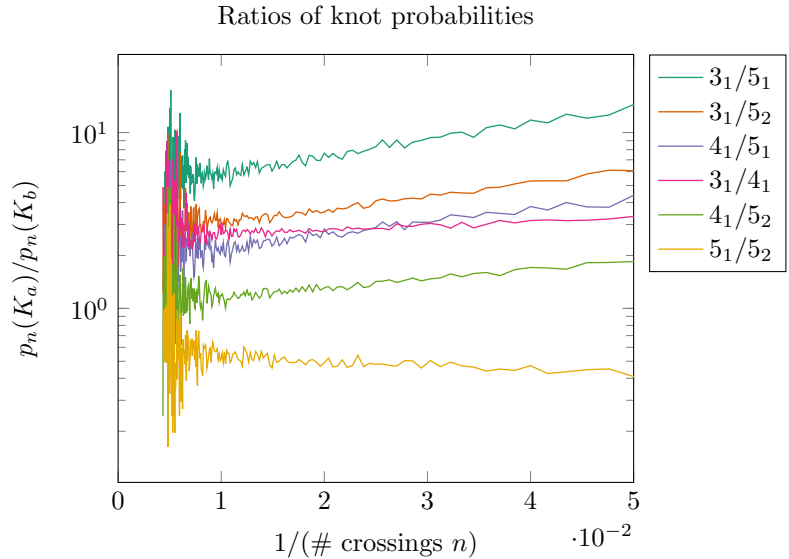

  \centering
  \includestandalone{knotrats}
  \caption{Ratios of probabilities of prime knot types in diagrams. Probabilities within each ratio are taken from independent runs of the MCMC sampler. Legend entries are sorted by their values at \(1/n = 0.05\).}
  \label{fig:primeknotrat}
\end{figure}

\begin{figure}
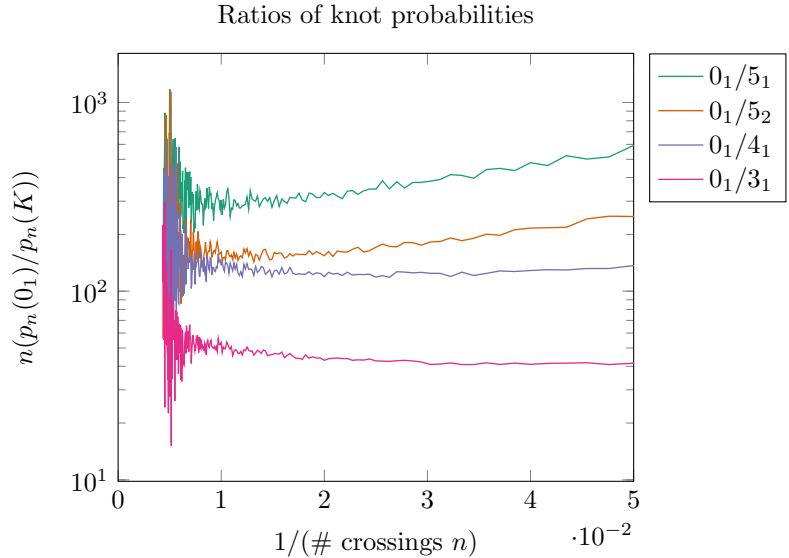

  \centering
  \includestandalone{unknotrats}
  \caption{Ratios of probabilities of the unknot with prime knots, with correction factor \(n\). Probabilities within each ratio are taken from independent runs of the MCMC sampler. Legend entries are sorted by their values at \(1/n = 0.05\).}
  \label{fig:unknotrats}
\end{figure}

\begin{figure}
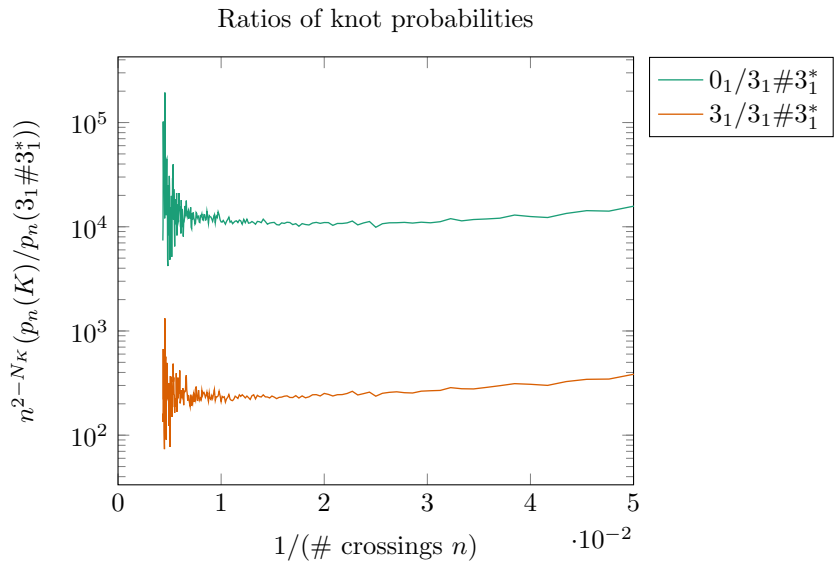

  \centering
  \includestandalone{squarerats}
  \caption{Ratios of probabilities of the square knots to trefoils and unknots (with corrections of \(n\) and \(n^2\) respectively. Probabilities within each ratio are taken from independent runs of the MCMC sampler. Legend entries are sorted by their values at \(1/n = 0.05\).}
  \label{fig:squarerats}
\end{figure}


\section{Conclusion}
\label{sec:conclusion}

We have described a new Markov chain Monte Carlo  method for sampling random plane curves efficiently. This then trivially extends to sample random knot diagrams by mapping vertices to crossings. This enables us to sample of large knot diagrams which are otherwise simply too rare to sample by rejection sampling methods. Due to the difficulty of tuning the Boltzmann MCMC to sample diagrams of a wide range of sizes, we modified our original Markov chain to use the flat histogram methods of Wang and Landau. This results in a chain that samples approximately uniformly across a wide range of sizes, and additionally gives estimates of the numbers of plane curves and knot diagrams. These estimated counts are in agreement with previously conjectured asymptotic forms.  Hence we conclude that our MCMC implementation can attain similar accuracy of the uniform rejection methods in far less time.  We have then tested the data from the Wang-Landau chain against data from a rejection sampler and find strong agreement across a wide range of statistics. 

Plane curves are a subset of 4-valent maps and it is not well undestood how different these two sets of objects are. With this in mind, we computed the average maximum face degree and face degree degree distributions and find significant differences.  One of the main aims of sampling random plane curves is to study random knotting. Our Wang-Landau sampler allows us to generate significant numbers of random knot diagrams of a range of sizes. We have then classified the knot types of those diagrams using HOMFLY-PT polynomials. This data allows us to conjecture that random knot diagrams have very similar asymptotic behaviour to other models of random knotting, such as self-avoiding polygons. These asymptotic forms are consistent with the idea that the knotted portion of a random knot diagram is quite localized. We also give evidence that the HOMFLY-PT polynomial software \parencite{Ewing1997} that we used struggles to compute invariants of complicated knots, but does succeed for simple knots, even when their embeddings might be very large.

\subsection{Future work}
\label{sec:otherobjects}

The Markov transitions presented here are based on ``shadows'' of the Reidemeister moves on \emph{knot diagrams}. It is thus natural to consider a Markov chain on knot diagrams generated by similar transitions corresponding to the proper Reidemeister moves, taking into account over-under signing of the diagram. This would produce a Markov process on the so-called Reidemeister graph~\parencite{Barbensi18}. Indeed, we expect this Markov chain to be ergodic (with ergodicity classes being fixed knot types), and we expect transitions to be of similar computational complexity. It would also have the advantage of not requiring a knot-classification step. We note, however, that Corollary~\ref{cor:capncross} fails in this case; it is known that there are diagrams who represent the same knot type, but whose transition paths all involve an increase in the number of crossings~\parencite{Kauffman2006}. To make matters worse, unlike the spaces of plane curves where the diameter has \(n^{3/2}\) growth, the upper bound on the diameters of the spaces of knot diagrams is far larger~\parencite{hass2001, Nowik2009, lackenby2015}. Hence in the case of knot diagrams, care must be taken to ensure that there are satisfactory parameters for the Markov chain to converge to the uniform distribution in a reasonable amount of time. It should also be noted that rejection sampling becomes even less efficient for sampling fixed knot types, since not only are knot diagrams exponentially rare in the space of 4-valent maps, but knot diagrams of a specific type are exponentially rare in the space of knot diagrams.

Beyond sampling knot diagrams with fixed knot type, the flat Reidemeister moves discussed in this paper also apply to planar immersions of any fixed number of circles. These diagrams are called \emph{link shadows}; the smallest such object is the unique 2-crossing 4-valent planar map of 2 link components. In this case, the Markov chain is still ergodic (the proof of Theorem~\ref{thm:ergodic} is not affected by the immersion having a different number of link components). Hence this technique could be used to sample large immersions of any fixed number of link components. We could also restrict or alter the transitions; for instance we could remove the shadow Reidemeister I move, whence the ergodicity classes of the Markov chain would be spherical curves of fixed spherical Whitney number~\parencite{Arnold1995,Nowik2009}. Further, by using Reidemeister moves instead of flat Reidemeister moves, one could also sample link diagrams of fixed link types.

\section*{Acknowledgements}

The authors would like to thank the Pacific Institute of Mathematics Sciences for funding the collaborative research group (CRG) on Applied Combinatorics. Part of this work started at a summer school on applied combinatorics funded by that CRG. The second author acknowledges funding from NSERC in the form of a Discovery Grant. The authors would also like to thank Chris Soteros, Jason Cantarella, and Stu Whittington for many helpful discussions.

\begingroup
\raggedright{}
\sloppy
\printbibliography{}
\endgroup

\appendix

\section{Complete detailed balance proofs}
\label{sec:detailedbalanceproofs}

\begin{proof}[Detailed balance equations for Boltzmann Markov chain; Theorem~\ref{thm:ergodic}]
  We check the detailed balance equations for each transition:
  \begin{enumerate}
  \item Suppose that \(N = \RIp(D,a)\) with root flag \(b\). This means that \(N\)
    is unique in that \(D = \RIm(N,b)\). Then
    \begin{align*}
      \Trprb(D \to N)\Prb(D) &= \Trprb(N \to D)\Prb(N) \\
      \frac {zp_1}{2}z^n &= \frac {p_1}{2}z^{n+1},
    \end{align*}
    so the equation holds.
    
  \item Suppose that \(N = \RIIp(D,a,a')\) with root flag \(b\). This means that
    \(N\) is unique in that \(D = \RIIm(N,b)\). The flags \(a,
    a'\) lie along a face in \(D\) of degree \(d\). Then
    \begin{align*}
      \Trprb(D \to N)\Prb(D) &= \Trprb(N \to D)\Prb(N) \\
      \frac{z^2p_2}{2(d-1)}z^n &= \frac{p_2}{2(d-1)}z^{n+2}.
    \end{align*}

  \item Suppose that \(N \ne \RIII(D,a)\) and that \(N\) is a re-rooting of \(D\).
    Then,
    \begin{align*}
      \Trprb(D \to N)\Prb(D) &= \Trprb(N \to D)\Prb(N) \\
      (1-(p_1+p_2+p_3))\frac{\Aut{D}}{4n}z^n &= (1-(p_1+p_2+p_3))\frac{\Aut{N}}{4n}z^n.
    \end{align*}

    Because \(D\) and \(N\) differ only by a re-rooting, their underlying number of automorphisms are the same; \(\Aut{D} = \Aut{N}\). Hence equality follows.


  \item Suppose that \(N = \RIII(D,a)\) has root \(b\) and that \(N\) is \emph{not} a re-rooting of \(D\). Then \(N\) is unique in that \(D = \RIII(N, b)\), so
    \begin{align*}
      \Trprb(D \to N)\Prb(D) &= \Trprb(N \to D)\Prb(N) \\
      z^np_3 &= z^np_3.
    \end{align*}

  \item If \(N = \RIII(D,a)\) has root \(b\) \emph{and} \(N\) is a re-rooting of \(D\), then the transition probabilities of the previous two cases are summed (as the different transitions are independent), so that
    \begin{align*}
      \Trprb(D \to N)\Prb(D)
      &= \Trprb(N \to D)\Prb(N) \\
      \left(p_3 + (1-(p_1+p_2+p_3))\frac{\Aut{D}}{4n}\right)z^n
      &= \left(p_3 + (1-(p_1+p_2+p_3))\frac{\Aut{N}}{4n}\right)z^n.\qedhere
    \end{align*}

  \end{enumerate}
\end{proof}

\begin{proof}[Proof of Corollary~\ref{cor:noreroot}]
  As this Markov chain can perform all flat Reidemeister transitions and achieve all curve rootings, this Markov chain explores the space of curves. It remains to show that detailed balance holds.

  For a pair of diagrams \(D, N\), let \(\Trprb_p(D\to N)\) denote the probability of transitioning from \(D\) to \(N\) under this modified Markov chain. Let \(\Trprb'_p(D\to N)\) be the probability of transitioning from \(D\) to \(N\) under the original Markov chain (no interstitial re-rooting). Finally, let \(\Trprb'_0(D\to N)\) be the probability of transitioning from \(D\) to \(N\) under the original Markov chain with all \(p_i = 0\) (only re-roots are performed).

  Notice that \(\Trprb_p(D \to N) = \sum_{B}\sum_{C}\Trprb'_0(C \to N)\Trprb'_p(B \to C)\Trprb'_0(D\to B)\), where the sums are over all rooted curves. Then,
  \begin{subequations}
  \begin{align}
    \Trprb_p(D \to N)\Prb(D)
    &= \sum_{B}\sum_{C}\Trprb'_0(C \to N)\Trprb'_p(B \to C)\Trprb'_0(D\to B)\Prb(D) \\
    &= \sum_{B}\sum_{C}\Trprb'_0(C \to N)\Trprb'_p(B \to C)\Trprb'_0(B\to D)\Prb(B) \\
    &= \sum_{B}\sum_{C}\Trprb'_0(C \to N)\Trprb'_p(C \to B)\Trprb'_0(B\to D)\Prb(C) \\
    &= \sum_{B}\sum_{C}\Trprb'_0(N \to C)\Trprb'_p(C \to B)\Trprb'_0(B\to D)\Prb(N) \\
    &= \Trprb_p(N \to D)\Prb(N),
  \end{align}
  \end{subequations}
  so detailed balance holds for the modified Markov chain, and hence it is ergodic.
\end{proof}

\begin{proof}[Detailed balance equations for Wang-Landau Markov chain; Theorem~\ref{thm:wlergodic}]
  The main concern now is that transitions must pass an additional Metropolis-Hastings check of \(\min\left\{ 1, g_n/g_m \right\}\). Note that for any \(g_n, g_m > 0\),
  \begin{equation} \frac{\min\left\{ 1, g_n/g_m \right\}}{\min\left\{ 1, g_m/g_n \right\}} = \frac {g_n}{g_m}. \end{equation}
  We check the detailed balance equations for each transition:
  \begin{enumerate}
  \item Suppose that \(N = \RIp(D,a)\) with root flag \(b\). This means that \(N\)
    is unique in that \(D = \RIm(N,b)\). Then
    \begin{align*}
      \Trprb(D \to N)\Prb(D)
      &= \Trprb(N \to D)\Prb(N) \\
      \min\left\{1, \frac{g_n}{g_{n+1}}\right\} \frac {p_1}{2} \frac{1}{Rg_n}
      &= \min\left\{1, \frac{g_{n+1}}{g_{n}}\right\}\frac {p_1}{2}\frac{1}{Rg_{n+1}} \\
      \frac{\min\left\{1, g_{n}/g_{n+1}\right\}}{\min\left\{1, g_{n+1}/g_{n}\right\}}\frac {p_1}{2} \frac{1}{Rg_n} 
      &= \frac {p_1}{2}\frac{1}{Rg_{n+1}} \\
      \frac{g_n}{g_{n+1}} \frac {p_1}{2} \frac{1}{Rg_n}
      &= \frac {p_1}{2}\frac{1}{Rg_{n+1}},
    \end{align*}
    so the equation holds.
  \item Suppose that \(N = \RIIp(D,a,a')\) with root flag \(b\). This means that
    \(N\) is unique in that \(D = \RIIm(N,b)\). The flags \(a,
    a'\) lie along a face in \(D\) of degree \(d\). Then
    \begin{align*}
      \Trprb(D \to N)\Prb(D)
      &= \Trprb(N \to D)\Prb(N) \\
      \max\left\{1, \frac{g_n}{g_{n+2}} \right\}\frac{p_2}{2(d-1)}\frac{1}{Rg_n}
      &= \max\left\{1, \frac{g_{n+2}}{g_{n}} \right\}\frac{p_2}{2(d-1)}\frac{1}{Rg_{n+2}}\\
      \frac{g_n}{g_{n+2}}\frac{p_2}{2(d-1)}\frac{1}{Rg_n}
      &= \frac{p_2}{2(d-1)}\frac{1}{Rg_{n+2}}.
    \end{align*}

  \item Suppose that \(N \ne \RIII(D,a)\) and that \(N\) is a re-rooting of \(D\).
    Then,
    \begin{align*}
      \Trprb(D \to N)\Prb(D) &= \Trprb(N \to D)\Prb(N) \\
      (1-(p_1+p_2+p_3))\frac{\Aut{D}}{4n}\frac{1}{Rg_n} &= (1-(p_1+p_2+p_3))\frac{\Aut{N}}{4n}\frac{1}{Rg_n}.
    \end{align*}

    Because \(D\) and \(N\) differ only by a re-rooting, their underlying number of automorphisms are the same; \(\Aut{D} = \Aut{N}\). Hence equality follows.


  \item Suppose that \(N = \RIII(D,a)\) has root \(b\) and that \(N\) is \emph{not} a re-rooting of \(D\). Then \(N\) is unique in that \(D = \RIII(N, b)\), so
    \begin{align*}
      \Trprb(D \to N)\Prb(D) &= \Trprb(N \to D)\Prb(N) \\
      \frac{p_3}{Rg_n} &= \frac{p_3}{Rg_n}.
    \end{align*}

  \item If \(N = \RIII(D,a)\) has root \(b\) \emph{and} \(N\) is a re-rooting of \(D\), then the transition probabilities of the previous two cases are summed (as the different transitions are independent), so that
    \begin{align*}
      \Trprb(D \to N)\Prb(D)
      &= \Trprb(N \to D)\Prb(N) \\
      \left(p_3 + (1-(p_1+p_2+p_3))\frac{\Aut{D}}{4n}\right)\frac{1}{Rg_n}
      &= \left(p_3 + (1-(p_1+p_2+p_3))\frac{\Aut{N}}{4n}\right)\frac{1}{Rg_n}.
    \end{align*}

\end{enumerate}
In all other cases, the transition probabilities are symmetrically zero. Hence we conclude that detailed balance holds with the hypothesized probability distribution.
\end{proof}

\end{document}
